\documentclass[11pt,a4paper]{article}

\usepackage{amssymb}
\usepackage{amsmath, amsthm}
\usepackage{hyperref}
\usepackage{graphicx, color}
\usepackage{setspace}

\def\RR{{\bf R}}
\def\ZZ{{\bf Z}}

\numberwithin{equation}{section}

\newcommand{\Conv}{\mathop{\rm Conv} }

\newcommand{\supp}{\mathop{\rm supp}}

\newtheorem{Thm}{Theorem}[section]
\newtheorem{Prop}[Thm]{Proposition}
\newtheorem{Lem}[Thm]{Lemma}
\newtheorem{Cor}[Thm]{Corollary}
\theoremstyle{definition}

\newtheorem{Rem}[Thm]{Remark}
\newtheorem{Ex}[Thm]{Example}

\title{A Nonpositive Curvature Property of Modular Semilattices}
\author{Hiroshi HIRAI \\
Department of Mathematical Informatics, \\
Graduate School of Information Science and Technology,   \\
The University of Tokyo, Tokyo, 113-8656, Japan.\\
\texttt{\normalsize hirai@mist.i.u-tokyo.ac.jp}}

\begin{document}

\maketitle
\begin{abstract}
	The orthoscheme complex of a graded poset is 
	a metrization of its order complex such that the simplex of each maximal chain is isometric to the Euclidean simplex of vertices $0, e_1,e_1+e_2,\ldots, e_1+e_2+ \cdots + e_n$. 
	This notion was introduced by Brady and McCammond in geometric group theory, and  has applications in discrete optimization and submodularity theory. 
	We address a question of what posets to yield 
	the orthoscheme complex having CAT(0) property.
	The orthoscheme complex of a modular lattice 
	is shown to be CAT(0), and it is conjectured 
	that this is the case for a modular semilattice.
	In this paper, we prove this conjecture affirmatively.
	This result implies that a larger class of weakly modular graphs yields CAT(0) complexes. 
\end{abstract}

Keywords: Modular semilattice, CAT(0) space, orthoscheme complex

\section{Introduction}
The {\em orthoscheme complex} $K({\cal P})$ of a graded poset ${\cal P}$ 
is a metrization of the order complex of ${\cal P}$ 
such that the simplex of each maximal chain 
is isometric to the Euclidean simplex of vertices
\begin{equation}\label{eqn:orthoscheme}
0,\ e_1,\ e_1+e_2,\ e_1+e_2+e_3,\ \cdots, e_1+e_2+\cdots+e_n,  
\end{equation}
where $e_i$ is the $i$-th unit vector in $\RR^n$ and $n$ is the length of the chain.
See Figure~\ref{fig:orthoscheme} for the construction.
\begin{figure} 
	\begin{center} 
		\includegraphics[scale=0.7]{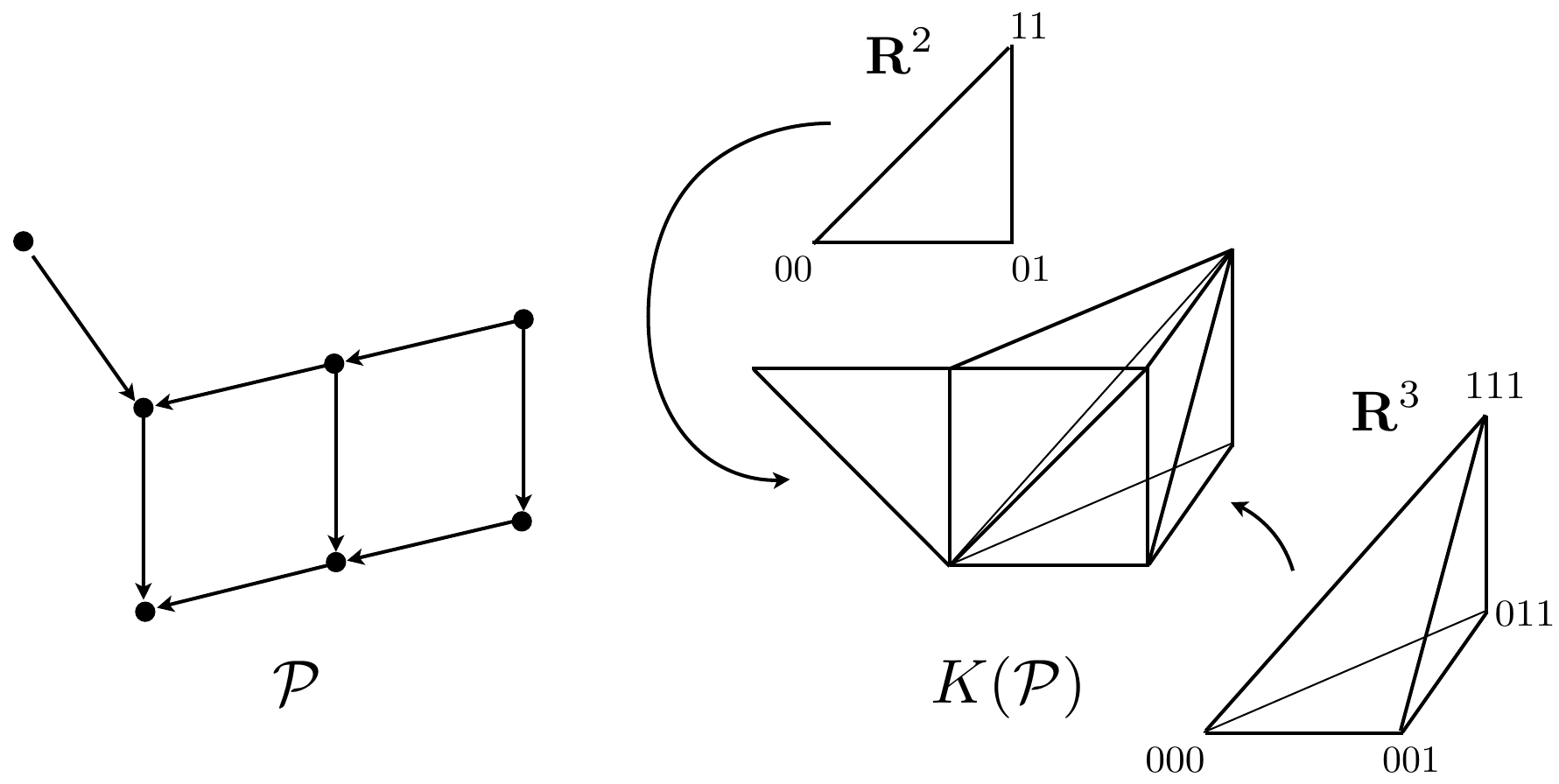}
		\caption{Orthoscheme complex}  
		\label{fig:orthoscheme}         
	\end{center}
\end{figure}
This concept was introduced by Brady and McCammond~\cite{BM10} 
in geometric group theory.
A central interest of the orthoscheme complex lies on the interplay between
combinatorial properties of ${\cal P}$ 
and metric properties of $K({\cal P})$. 
We are particularly interested in the situations where $K({\cal P})$ 
becomes a nonpositively curved metric space, i.e., CAT(0) space.
Here a {\em CAT(0) space} is a geodesic metric space in which  
every geodesic triangle is  {\em not thicker} than the 
corresponding triangle in Euclidean plane; 
see~\cite{BH} for the precise definition. 
In~\cite{BM10},  
Brady and McCammond made a beautiful conjecture saying 
that if ${\cal P}$ is the noncrossing partition lattice, 
then $K({\cal P})$ is CAT(0).
If this conjecture is true, 
then any braid group is a CAT(0) group, 
which is a longstanding conjecture in group theory.
See Haettel,  Kielak, and Schwer~\cite{HKS17} 
for the current best result on this direction.

Apart from such group-theory motivation, it is 
interesting to investigate 
which poset ${\cal P}$
has CAT(0) orthoscheme complex $K({\cal P})$.
If ${\cal P}$ is a Boolean lattice, 
then $K({\cal P})$ 
is isometric to a Euclidean cube $[0,1]^n$ and hence is CAT(0). 
This fact is naturally generalized to distributive lattices:
If ${\cal P}$ is a distributive lattice, 
then $K({\cal P})$ is isometric to
the {\em order polytope} (of the subposet of join-irreducibles in ${\cal P}$), 
and hence is CAT(0).
Brady and McCammond~\cite{BH} conjectured 
that this CAT(0) property holds for modular lattices.
Haettel, Kielak, and Schwer~\cite{HKS17} proved 
this conjecture for complemented modular lattices.
Then Chalopin et al.~\cite{CCHO14} 
proved the general case.
\begin{Thm}[\cite{CCHO14}]\label{thm:modularlattice}
	If ${\cal L}$ is a modular lattice of finite rank, then $K({\cal L})$ is CAT(0).
\end{Thm}

Chalopin et al.~\cite{CCHO14} further 
studied the orthoscheme complexes of (meet-)semilattices.
They proved that the orthoscheme complex of 
a {\em median semilattice}, a semilattice 
analogue of a distributive lattice, is CAT(0), and 
conjectured the CAT(0) property for {\em modular semilattices}~\cite[Conjecture 7.3]{CCHO14}; 
see also \cite[Problem 6.10]{Bacak}. 
%
%
The main result in this paper affirmatively solves
this conjecture.
\begin{Thm}\label{thm:main}
	If ${\cal L}$ is a modular semilattice of finite rank, 
	then $K({\cal L})$ is CAT(0).
\end{Thm}

This theorem can be used to show the CAT(0) property 
for a larger class of orthoscheme 
complexes related to {\em weakly modular graphs}.
This was a motivation of~\cite{CCHO14} 
to consider modular semilattices.
A weakly modular graph $G$ is a connected undirected graph satisfying 
the triangle condition (TC) and quadrangle condition (QC):
\begin{itemize}
	\item[(TC)] For vertices $x,y,z$ with $d(x,y) = 1$ and $d(x,z) = d(y,z) = k > 0$, there is a vertex $v$ with $d(x,v) = d(y,v)=1$ and  $d(v,z) = k-1$. 
	\item[(QC)] For vertices $u,x,y,z$ with $d(u,x) = d(u,y) = 1 = d(x,y) - 1$ and $d(x,z) = d(y,z) = d(u,z) - 1 = k>0$, there is a vertex $v$ with $d(x,v) = d(y,v)=1$ and  $d(v,z) = k-1$.
\end{itemize}
Here $d$ denotes the shortest path metric of $G$.
A {\em sweakly modular graph}, {\em swm-graph} for short, 
is a weakly modular graph 
such that there is no induced $K_4^-$-subgraph 
and no isometric $K_{3,3}^-$-subgraph, 
where $K_4^-$ and $K_{3,3}^-$ are the graphs 
obtained from $K_4$ and $K_{3,3}$, respectively, by removing one edge.
It turned out in \cite[Chapters 6--8]{CCHO14} that swm-graphs 
constitute a particularly nice subclass of weakly modular graphs 
and have rich connections to nonpositively curved spaces.
Examples of swm-graphs include 
median graphs, the covering graphs 
of modular (semi)lattices, dual polar graphs, and 
a certain subgraph of the 1-skeleton of a Euclidean building of type C. 
A recent work \cite{HH17_building} shows that 
Euclidean buildings of type A can also be characterized by certain modular lattices and hence swm-graphs.
All these examples are connected 
to CAT(0) spaces: 
A median graph is precisely 
the 1-skeleton of a CAT(0) cubical complex~\cite{Chepoi00}.
Also it is well-known that 
a Euclidean building canonically admits 
a CAT(0) metric. 

Chalopin et al.\cite{CCHO14} 
presented a general construction 
of a metrized simplicial complex $K(G)$ from any swm-graph $G$, 
which generalizes the construction of the CAT(0) cubical complex 
from a median graph 
and explains the recovery of a Euclidean building of type A/C from its graph.
The construction of $K(G)$ is briefly explained as follows.
A {\em Boolean-gated set} of an swm-graph $G$ is a 
nonempty vertex subset $X$ having the following properties:
\begin{itemize}
	\item For any distinct vertices $x,y \in X$, any common neighbor of $x,y$ 
	belongs to $X$.
	\item For any $x,y \in X$ with $d(x,y) = 2$, there are common neighbors $u,v$ of $x,y$ 
	with $d(u,v) =2$. 
\end{itemize}
The family of all Boolean-gated set forms 
a graded poset in terms of the reverse inclusion order, 
where singletons are maximal elements, and the maximum length of a chain 
is called the {\em cube-dimension} of $G$.
Then one can consider the orthoscheme complex $K(G)$ of this poset.
Chalopin et al.~\cite{CCHO14} conjectured 
that $K(G)$ is CAT(0) (provided $G$ has a finite cube-dimension), 
which is one of the main conjectures of that paper.

The above Theorem~\ref{thm:main} implies this conjecture.
Indeed, each point in $K(G)$ has a convex neighborhood isometric 
to the orthoscheme complex of some modular semilattice; see \cite[Proposition 8.3]{CCHO14}.
Thus, by Theorem~\ref{thm:main}, 
$K(G)$ is locally CAT(0).
Also it is shown \cite[Theorem 8.1(iii)]{CCHO14} that 
$K(G)$ is a contractible complex. 
By the Cartan-Hadamard theorem~\cite[Chapter II.4]{BH}, $K(G)$ is (globally) CAT(0):  
\begin{Cor}\label{cor:swm}
	If $G$ is an swm-graph of finite cube-dimension, 
	then $K(G)$ is CAT(0).
\end{Cor}

\paragraph{Related work.}
Although the main source of this paper is \cite{CCHO14}, 
our primary motivation and proof idea of Theorem~\ref{thm:main} come from  
recent developments in 
algorithmic theory on CAT(0) spaces.
One of the starting points is the {\em space of phylogenetic trees} 
(also called  the {\em BHV tree space}) due to 
Billera, Holmes, and Vogtmann~\cite{BHV01}, 
which parametrizes 
all weighted trees having a given set (taxa) as leaves.
This space is a (non-convex) polyhedral region 
obtained by gluing nonnegative orthants in Euclidean space, 
and admits a natural length metric.
They showed that the BHV tree space 
is CAT(0).
By the unique geodesic property of CAT(0) spaces,  
a geodesic between two phylogenetic trees 
is uniquely determined, which will be a useful comparison tool 
for phylogenetic trees. 
This gives rise to a natural computational problem 
to find the unique geodesic and to determine its complexity.
A wonderful solution was given by Owen~\cite{Owen11} and Owen and Provan~\cite{OwenProvan11}:
The former paper established an explicit formula of
 tree-space geodesics, and,  based on it, the latter paper
gave a polynomial time algorithm to find the geodesics.
Subsequently,  Miller, Owen, and Provan~\cite{MOP15} generalized 
these results to {\em CAT(0) orthant spaces}, 
which are CAT(0) spaces obtained 
by gluing nonnegative orthants in Euclidean space.

Abram and Ghrist~\cite{AbramGhrist04}  
formulated the state space of a robot as a (locally) CAT(0) cubical complex, in which an optimal motion plan between two states is obtained 
by a geodesic in this space.
Motivated by this application, 
Ardila, Owen, and Sullivant~\cite{ArdilaOwenSullivant2012} 
studied the geodesic problem in general CAT(0) cubical complexes, and 
developed a compact representation of CAT(0) cubical complexes and
an algorithm to find geodesics. 
However this algorithm is not guaranteed to be polynomial.  
For the challenge of a polynomial time algorithm,
Hayashi~\cite{Hayashi} gave a satisfactory solution 
by developing a simple polynomial time algorithm to
find a ``near" geodesic with accuracy parameter $\epsilon > 0$, 
where 
$\log \epsilon^{-1}$ is a part of the input.

Theory of 
convex optimization on CAT(0) spaces is 
a new emerging field, 
in which several Euclidean optimization algorithms have been being extended~\cite{BacakBook}; see also \cite{Bacak}.
The construction ${\cal L} \hookrightarrow K({\cal L})$ gives rise to
a continuous relaxation, analogous to $\ZZ^n \hookrightarrow \RR^n$, of a discrete optimization problem on ${\cal L}$.
By using Theorem~\ref{thm:modularlattice}, 
Hirai~\cite{Hirai_L-convex} showed  
that any {\em submodular function}~\cite{FujiBook} 
on a modular lattice ${\cal L}$ can be extended to 
a convex function on the CAT(0) space $K({\cal L})$.
Hamada and Hirai~\cite{HamadaHirai} 
applied this result to 
a certain discrete optimization problem 
on a modular lattice ${\cal L}$ of all vector subspaces, 
and developed a polynomial time algorithm via 
a continuous optimization method for 
the CAT(0) space relaxation $K({\cal L})$.
In \cite{HH0ext,Hirai_L-convex}, 
modular semilattices, swm-graphs and related structures deserve
the underlying spaces of {\em discrete convex functions}, which
have played important roles to design efficient polynomial time algorithms 
to some classes of combinatorial optimization problems; 
see also \cite{HH_survey}.
 
\paragraph{Outline.}
We outline the proof of Theorem~\ref{thm:main} and 
the structure of this paper. 
In general, proving the CAT(0) property is not easy; 
we are currently in the position of seeking proof techniques
that are applicable to combinatorially-defined geodesic metric spaces. 
One successful tool is 
Gromov's combinatorial characterization of CAT(0) cubical complexes.
Indeed the CAT(0) property of the BHV tree space 
is an immediate consequence of this characterization.
Gromov's characterization can be proved by verifying 
the {\em link condition} of the link complex of each vertex; 
see the proof of \cite[Theorem 5.18]{BH}.
Our first attempt for proving Theorem~\ref{thm:main}  
was to adapt this argument. 
However we could not succeed.
Instead, we show the unique geodesic property of $K({\cal L})$. 
This is another equivalent condition 
of the CAT(0) property for a class of complexes~\cite[Theorem 5.4]{BH}, 
which includes our complexes.
In fact, the unique geodesic property 
of the BHV tree space can directly be proved,  
without knowing the CAT(0) property, 
from the formula of geodesics.
The orthoscheme complexes of modular semilattices
can realize BHV tree spaces as well as CAT(0) orthant spaces.
The central of our proof is to extend Owen's geodesic 
formula to $K({\cal L})$.
We construct various 
nonexpansive retractions in $K({\cal L})$ in lattice-theoretic ways, 
and show 
that any geodesic between two points $x,y \in K({\cal L})$
belongs to a certain subcomplex of $K({\cal L})$ determined by $x,y$. 
This subcomplex is a variant (a {\em median orthoscheme complex}) 
of CAT(0) orthant spaces.
By extending Owen's formula, 
we show that there uniquely exists a geodesic in this subcomplex, which establishes the unique geodesic property.

This paper is organized as follows.
In Section~\ref{sec:preliminaries},  we introduce necessary backgrounds 
on geodesic metric spaces (Section~\ref{subsec:geodesicmetricspace}) and formally introduce orthoscheme complexes (Section~\ref{subsec:orthoscheme}).
In Section~\ref{subsec:cubical} 
we explain Owen's geodesic formula with a new perspective, and outline how to prove 
the unique geodesic property from this formula, 
which is the underlying proof 
idea of Theorem~\ref{thm:main}.
In Section~\ref{subsec:median}, we extended 
this result for a median orthoscheme complex, which is
the orthoscheme complex of a median semilattice.
In Section~\ref{sec:modularsemilattice}, 
we study the orthoscheme complex of a modular semilattice 
and prove Theorem~\ref{thm:main}, 
where a detailed proof outline is given in the beginning of the section. 
Our proof is constructive, and sheds an insight 
on the geodesic problem from an algorithmic point of view
(Remarks~\ref{rem:MSIP} and \ref{rem:MVSP}).

The extended abstract of this paper will appear in the proceedings 
of the 11th Hungarian-Japanese Symposium 
on Discrete Mathematics and Its Applications (May 27--30, 2019).
Corollary~\ref{cor:swm} was announced at Geometry Seminar 
in University of Wroclaw at September~3,~2015.

\section{Preliminaries}\label{sec:preliminaries}

Let $\RR$ denote the set of real numbers.
For a function $x: S \to \RR$, the {\em nonzero support} of $x$ is defined by
\[
\supp x := \{ v \in S \mid x(v) \neq 0 \}.
\]
For a set $S$ and a subset $R \subseteq \RR$, 
let $R^S$ denote the set of all functions $x:S \to R$
such that its nonzero support $\supp x$
is finite.
An element $x$ of $R^S$ is written as a formal sum
\begin{equation*}
x = \sum_{v \in S} x_v v,
\end{equation*}
where $x(v) = x_v$ for $v \in S$.
For a subset $S' \subseteq S$,
the restriction of $x \in R^S$ to $S'$ is denoted by $x|_{S'}$.
Any element $x \in R^{S'}$ is naturally regarded as $x \in R^S$ by $x(v) := 0$ for $v \in S \setminus S'$. 
In particular, $R^{S'} \subseteq R^S$.

\subsection{Geodesic metric space}\label{subsec:geodesicmetricspace}

Let $K$ be a metric space with distance function $d = d_K$.
A path $P$ is a continuous function from $[0,1]$ to $K$.
If $P(0) = x$ and $P(1) = y$, then we say 
that $P$ connects $x$ and $y$ or $P$ is an $(x,y)$-path.
If the image of $P$ belongs to a subset $K'$ of $K$, 
then we simply say that $P$ belongs to $K'$.
The {\em length} $d(P)$ of a path $P$ is defined by
\begin{equation}\label{eqn:length}
\sup \sum_{i = 1}^{N} d(P(t_{i-1}),P(t_{i})),
\end{equation}
where the supremum is taken over all $N > 0$ and $0 = t_0 < t_1 < \dots < t_N =1$.
It is obvious that $d(x,y) \leq d(P)$ for any $(x,y)$-path $P$.
A metric space $K$ is called {\em geodesic} 
if for every $x,y \in K$ there is an $(x,y)$-path $P$ with $d(P) = d(x,y)$; 
such a path $P$ is called {\em shortest}. 
A {\em geodesic} between $x$ and $y$ is 
a shortest $(x,y)$-path $P$ 
proportionally parametrized by its length. Namely, a geodesic is a path $P$ satisfying 
$|s-t| d(P(0),P(1)) = d(P(s),P(t))$ for $s,t \in [0,1]$.
For a subset $K'$ of a geodesic metric space $K$, 
the metric $d_{K'}$ of $K'$ is defined by the infimum of the length of a path connecting two points in $K'$, 
where the length is measured in the metric $d_K$ on $K$ as in (\ref{eqn:length}).
The resulting metric space $K'$ is said to be a {\em subspace} of $K$. 
By definition, $d_{K'}(x,y) \geq d_{K}(x,y)$ holds for all $x,y \in K'$.
A subspace $K'$ is said to be {\em convex} if $d_{K'}(x,y) = d_{K}(x,y)$ holds for all $x,y \in K'$. 
In addition, $K'$ is called {\em strictly-convex} if 
for every $x,y \in K'$ every shortest $(x,y)$-path belongs to $K'$. 
A (continuous) map $\phi:K \to M$ for metric spaces $K,M$ is said 
to be {\em nonexpansive} if $d_M(\phi(x), \phi(y)) \leq d_K(x,y)$ for all $x,y \in K$.
In this paper, we will often face a nonexpansive {\em retraction} $\phi: K \to K$, 
i.e., $\phi$ is the identity on $\phi(K)$.
In this case, the retract $\phi(K)$ is a convex subspace of $K$.

We next introduce a CAT(0) space.
We only give the following simpler definition, which is not used later.
A geodesic metric space $K$ is said to be {\em CAT(0)} 
if for every point $x \in K$ and every geodesic $P$, 
the function $t  \mapsto d(x, P(t))^2$ is $1$-strongly convex, i.e.,
\begin{equation}\label{eqn:CAT(0)}
d(x,P(t))^2 \leq (1-t) d(x,P(0))^2 + td(x,P(1))^2 - t(1-t)d(P(0),P(1))^2.
\end{equation}
A CAT(0) space is {\em uniquely geodesic} in the sense that for every pair of points 
there exists a unique geodesic connecting them.
This property characterizes the CAT(0) property 
for a larger class of geodesic metric spaces.
An {\em $M_0$-polyhedral complex} is a metric space obtained 
by gluing convex polytopes in Euclidean space along their common isometric faces.
The precise definition is given in \cite[Chapter I.7]{BH}.
It is known~\cite[Theorem 7.19]{BH} that an $M_0$-polyhedral complex is a complete geodesic metric space 
if it is constructed from a family of convex polytopes 
in which there are finitely many isometry classes in the family.
\begin{Lem}[{\cite[Theorem 5.4]{BH}}]\label{lem:CAT(0)=uniquegeodesic}
	Let $K$ be an $M_0$-polyhedral complex with finite isometry types of cells.
	Then $K$ is CAT(0) if and only if $K$ is uniquely geodesic.
\end{Lem}
In this paper we only deal with  $M_0$-polyhedral complexes $K$
with finite isometry types of simplices.  The above characterization is applicable.
Hence, instead of (\ref{eqn:CAT(0)}),  
we mainly concern the unique geodesic property.
From a general result~\cite[Corollary 7.29]{BH},
each geodesic in $K$ is {\em polygonal} 
in the sense that it meets a finite number of simplices.
Therefore we can assume that a path in $K$ 
is polygonal if necessary.
To show that $K' \subseteq K$ is strictly-convex, 
it suffices to construct a nonexpansive retraction $\phi$ to $K'$ such that
$d_K(\phi(P)) < d_K(P)$ holds for any polygonal path $P$ connecting any $x,y \in K'$ and not belonging to $K'$.
Such a  $\phi$ is particularly called a 
{\em strictly-nonexpansive retraction}.

We present one simple and useful lemma for proving the unique geodesic property.
For two (geodesic) metric spaces $M,N$, 
the product $M \times N$ is a (geodesic) metric space, 
where the distance function $d_{M \times N}$ is defined by
\[
d_{M \times N}((x,y),(x',y'))^2 = d_M(x,x')^2 + d_N(y,y')^2.
\]
For paths $P$ in $M$ and $Q$ in $N$, 
the {\em product} $(P,Q)$ of $P,Q$ is the path in $M \times N$ 
defined by $(P,Q)(t) := (P(t),Q(t))$ $(t \in [0,1])$. 
Note that (the image of) $(P,Q)$ depends on the parameterizations of $P,Q$.  
\begin{Lem}\label{lem:product}
	Let $M,N$ be metric spaces, and $K$ a subspace of $M \times N$.
	Let $P = (Q,R)$ be a path in $K$.
	Then it holds
	\begin{equation}\label{eqn:PQR}
	d_{K}(P)^2 \geq d_M(Q)^2 + d_N(R)^2.
	\end{equation}
	 If $Q$ and $R$ are (unique) geodesics in $M$ and in $N$, respectively, 
	 then the equality holds in (\ref{eqn:PQR}) and
	 $P$ is a (unique) geodesic in $K$.
\end{Lem}
\begin{proof}
	The latter statement follows from 
	a general property \cite[Proposition 5.3 (3)]{BH} 
	that $(Q,R)$ is a geodesic in 
	$M \times N$ if and only if $Q$ and $R$ are geodesics in $M$ and in $N$, respectively. 	
	We show the inequality (\ref{eqn:PQR}).
	Consider subdivision $0 = t_0 < t_1 < \cdots < t_m = 1$ of $[0,1]$.
	Define points $\gamma_k := (\sum_{i=1}^k d_M(Q(t_{i-1}),Q(t_{i})), \sum_{i=1}^k d_N(R(t_{i-1}),R(t_{i})))$ in $\RR^2$
	for $k=0,1,2,\ldots,m$. 
	Consider the polygonal path $\gamma$ 
	in $\RR^2$ obtained by connecting these points $\gamma_k$
	by segments $[\gamma_{k-1},\gamma_k]$.
	The length $d_{\RR^2}(\gamma)$ of the path $\gamma$ measured in Euclidean plane $\RR^2$ 
	is equal to $\sum_{i=1}^m d_{M \times N} (P(t_{i-1}), P(t_i)) (\leq d_K(P))$.
    By choosing a sufficiently fine subdivision of $[0,1]$, the end point
	$\gamma_m = \gamma(1)$ is arbitrarily close to $(d_M(Q), d_N(R))$.
	Namely, for arbitrary $\epsilon > 0$, we can choose a subdivision of $[0,1]$
	such that $\|\gamma(1) - \gamma(0)\|^2 \geq d_M(Q)^2 + d_N(R)^2 - \epsilon$.
	Thus we have 
	$d_K(P)^2 \geq d_{\RR^2}(\gamma)^2 \geq \|\gamma(1) - \gamma(0)\|^2 \geq d_M(Q)^2 + d_N(R)^2 - \epsilon$.
	Since $\epsilon >0$ was arbitrary, we have (\ref{eqn:PQR}).
\end{proof}
We will use this lemma in the following way: 
Given geodesics $Q$ and $R$ (that are easily obtained), 
if their product $(Q,R)$ belongs to $K$ (luckily), 
then $(Q,R)$ is a geodesic in $K$. 
\subsection{CAT(0) rooted cubical complex}\label{subsec:cubical}
A {\em cubical complex} is an $M_0$-polyhedral complex obtained 
by gluing Euclidean cubes of various dimensions.
We here consider a cubical complex 
in which all cubes contains a common vertex (root).
Such a cubical complex is naturally associated  
with an abstract simplicial complex as follows.
Let $V$ be a set. 
Consider $[0,1]^V$, i.e., the set of 
formal combinations of elements in $V$ for which 
the coefficients belong to $[0,1]$. 
Define metric $d$ on $[0,1]^V$ by the $l_2$-distance:
\begin{equation*}
d(x,y) := \sqrt{\sum_{v \in V} (x_v - y_v)^2} \quad (x,y \in [0,1]^V),
\end{equation*}
where the sum is a finite summation over $\supp x \cup \supp y$.
Let ${\cal S} \subseteq 2^V$ be an abstract simplicial complex on $V$, 
i.e., $S \subseteq S' \in {\cal S}$ implies $S \in {\cal S}$.
Suppose that the maximum cardinality of members in ${\cal S}$ is bounded by some constant.  
Let $K({\cal S}) \subseteq [0,1]^V$ be the subspace of all points $x$ with  
$\supp x \in {\cal S}$.
Namely $K({\cal S})$ is the union $\bigcup_{S \in {\cal S}}[0,1]^{S}$
of all Euclidean cubes $[0,1]^S$ over $S \in {\cal S}$.
Then $K({\cal S})$ is a complete geodesic space, and is called the 
{\em rooted cubical complex} associated with ${\cal S}$.
If $[0,1]$ is replaced by the set $\RR_+$ of nonnegative reals,
then the resulting complex is an {\em orthant space} 
in the sense of \cite{MOP15}. 
A rooted cubical complex is a strictly-convex subspace 
of the corresponding orthant space; 
consider strictly-nonexpansive retraction $\sum_v x_v v \mapsto \sum_v \min \{x_v, 1\} v$. 
Then results for rooted cubical complexes 
are easily adapted for orthant spaces, and vice versa. 
 
We are interested in CAT(0) rooted cubical complexes.
Then Gromov's characterization on CAT(1) all-right spherical complexes is rephrased as follows: 
\begin{Thm}[{Gromov; see \cite[Theorem 5.18]{BH}}]\label{thm:Gromov}
	For an abstract simplicial complex ${\cal S}$,
	the rooted cubical complex $K({\cal S})$ is CAT(0) if and only if 
	${\cal S}$ is a flag complex.
\end{Thm}
Here ${\cal S}$ is called a {\em flag complex} 
if it satisfies the flag condition (FL):
\begin{itemize}
	\item[(FL)] $S \subseteq V$ belongs to ${\cal S}$ if and only if 
	every $2$-element subset of $S$ belongs to ${\cal S}$. 
\end{itemize}
Notice that a flag  complex ${\cal S}$ is precisely 
the simplicial complex of stable sets in a graph 
(with vertex set $V$ and edge set $E := \{ uv \mid \{u,v\} \not \in {\cal S} \}$).
Thus we obtain a CAT(0) rooted cubical complex from an arbitrary undirected graph $G$, 
which is denoted by $K(G)$.

We next study geodesics in $K = K(G)$ 
and explain Owen's formula of geodesics, 
which was originally obtained for the BHV tree spaces and extended to 
CAT(0) orthant spaces by  Miller, Owen and Provan~\cite{MOP15}.
We first consider the special case where $G$ is a finite bipartite graph 
with two color classes $B,C$.
We also suppose that $G$ has no isolated vertices. 
We explain in (R1) and (R2) below how to reduce 
geodesics in $K(G)$ for general $G$
to this special case.

An {\em arch} is a sequence $(B = U_0,U_1,\ldots,U_m = C)$ of stable sets in $G$
such that 
\begin{equation}\label{eqn:arch}
B \cap U_{i-1} \supset B \cap U_{i}, \quad C \cap U_{i-1} \subset C \cap U_{i} \quad (i=1,2,\ldots,m),
\end{equation}
where $\subset$ means proper inclusion.
The {\em path space} $K({\cal A})$ 
relative to an arch ${\cal A} = (B = U_0,U_1,\ldots,U_m = C)$ is 
the subcomplex of $K$ consisting of cubes for $U_i$ $(i =0,1,2,\ldots,m)$.
Namely $K({\cal A}) := \bigcup_{i=0}^m [0,1]^{U_i}$.
%
A {\em path-space geodesic} is a geodesic in some path space $K({\cal A})$. 

Let $x,y$ be points in $K$ with $\supp x = B$ and $\supp y =C$.
We consider path-space geodesics connecting $x,y$.
Let ${\cal A}$ be an arch $(B = U_0,U_1,\ldots,U_m = C)$.
Define $X_i,Y_i$ $(i=1,2,\ldots,m)$ by
\begin{eqnarray}
	X_i &:=& (B \cap U_{i-1}) \setminus (B \cap U_{i}), \label{eqn:X_i}\\
	Y_i &:=& (C \cap U_{i}) \setminus (C \cap U_{i-1}). \label{eqn:Y_i}
\end{eqnarray}
Then $U_k = X_m \cup X_{m-1} \cup \cdots \cup X_{k+1} \cup Y_{k} \cup Y_{k-1} \cdots \cup Y_{1}$.
Also define positives $\|X_i\|$ and $\|Y_i\|$ by 
\begin{eqnarray}
	\|X_i\|^2 & := & \sum_{b \in X_i} x^2_b, \label{eqn:|X_i|}\\
	\|Y_i\|^2 & := & \sum_{c \in Y_i} y^2_c, \label{eqn:|Y_i|}\
\end{eqnarray}
where $x = \sum_{b \in B} x_b b$ and $y = \sum_{c \in C} y_c c$.
Then we define $v({\cal A};x,y)$ by
\begin{equation}\label{eqn:v(A;x,y)}
	v({\cal A};x,y)  := \sqrt{\sum_{i=1}^m (\| X_i\| + \|Y_i\|)^2}.
\end{equation}
In fact, this quantity $v({\cal A};x,y)$ is a lower bound of  
the distance between $x,y$ in the path space $K({\cal A})$.
To see this fact, consider the complex $K({\cal A},X_i,Y_i) := K({\cal A})|_{X_i \cup Y_i}$ obtained by projecting $K({\cal A})$ 
to $[0,1]^{X_i \cup Y_i}$. Then $K({\cal A},X_i,Y_i)$ 
is the union of two cubes 
$[0,1]^{X_i}$ and $[0,1]^{Y_i}$ sharing exactly one point (the origin $0$).
Therefore the unique geodesic between $x|_{X_i \cup Y_i}$ and $y|_{X_i \cup Y_i}$
goes on the union of segments 
$[x|_{X_i \cup Y_i}, 0]$ and $[0, y|_{X_i \cup Y_i}]$ and has
length $\| X_i\| + \|Y_i\|$.
Since $X_i,Y_i$ $(i=1,2,\ldots,n)$ form a partition of $V$, 
the path space $K({\cal A})$ is viewed as a subspace of 
the product of $K({\cal A},X_i,Y_i)$ over $i=1,2,...,m$.
Thus, by Lemma~\ref{lem:product}, 
the length of any path connecting $x,y$ is bounded below by $v({\cal A};x,y)$.  

This motivates a condition for arch ${\cal A}$ to attain this lower bound, 
or equivalently, to lift these projected geodesics to a geodesic in $K({\cal A})$.
An arch $(B = U_0,U_1,\ldots,U_m = C)$ is said to be {\em $(x,y)$-concave} 
if it holds
\begin{equation}\label{eqn:concave}
	\frac{\|X_1\|}{\|Y_1\|} < \frac{\|X_2\|}{\|Y_2\|} < \cdots < \frac{\|X_m\|}{\|Y_m\|}.
\end{equation}
We see below 
that this condition is enough to such a lifting.
Before that, we give an alternative geometric interpretation of (\ref{eqn:concave}), which explains the meaning of terms ``arch" and ``concave."
Plot the points 
\begin{equation*}
	\xi_k :=  (\sum_{b \in U_k \cap B} x_b^2, \sum_{c \in U_k \cap C} y_c^2 ) = (\sum_{i=k+1}^m \|X_{i}\|^2,  \sum_{i=1}^{k} \|Y_i\|^2) \quad (k=0,\ldots,m)
\end{equation*}
in the plane $\RR^2$.
Draw a line between each of consecutive points $\xi_k,\xi_{k+1}$.
See Figure~\ref{fig:arch}.
\begin{figure} 
	\begin{center} 
		\includegraphics[scale=0.7]{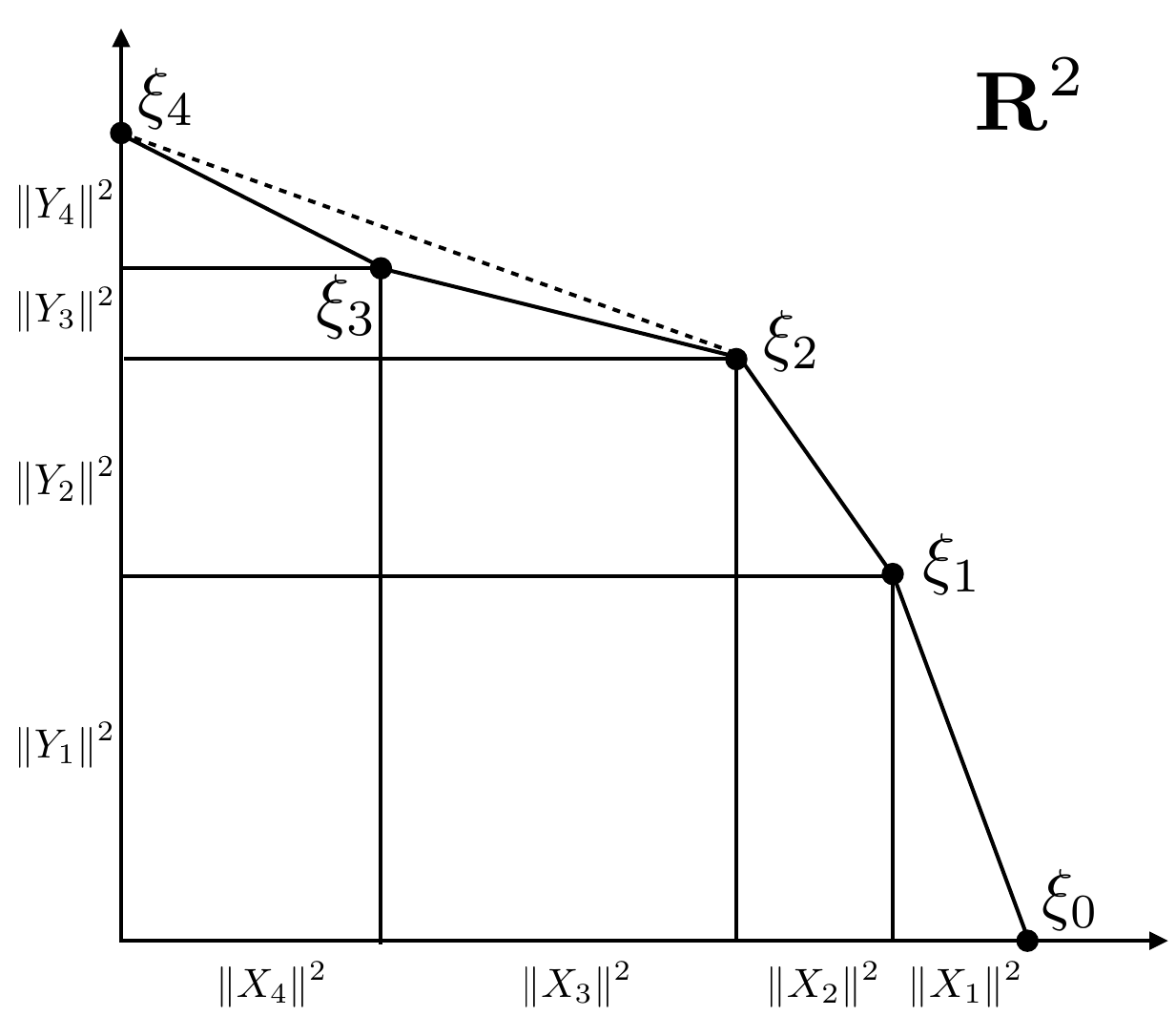}
		\caption{Polygonal representation of an arch with $\| X_1 \|/\|Y_1\| < \|X_2 \| / \|Y_2\| < \| X_3 \|/\|Y_3\| > \|X_4 \| / \|Y_4\|$}  
		\label{fig:arch}         
	\end{center}
\end{figure}

Then the condition (\ref{eqn:concave}) is equivalent to 
the condition that 
the resulting polygonal curve forms 
a concave arch with bending points $\xi_k$, or equivalently, that 
each $\xi_k$ is an extreme point of 
the convex hull of $(0,0)$ and $\xi_k$ $(k=0,1,2,\ldots,m)$.
\begin{Thm}[{\cite{Owen11}; also see \cite[Theorems 2.3 and 2.4]{OwenProvan11}}]\label{thm:owen}
	\begin{itemize}
		\item[{\rm (1)}]
	For an $(x,y)$-concave arch ${\cal A} = (B = U_0,U_1,\ldots,U_m = C)$, 
	the unique geodesic $t \mapsto 
	\sum_{b \in B} x_b(t) b + \sum_{c \in C} y_c(t) c$ 
	in the path space $K({\cal A})$
	is given by
	\begin{eqnarray}
		&& x_b(t) = \left\{
		\begin{array}{ll} \displaystyle
			 \left( 1 - t \frac{\|X_i\| + \|Y_i\| }{\|X_i\|} \right)x_b & \displaystyle
			{\rm if}\ 0 \leq t \leq \frac{\|X_i\|}{\|X_i\| + \|Y_i\|}, \\[1em]
			0 & \displaystyle {\rm if}\ \frac{\|X_i\|}{\|X_i\| + \|Y_i\| } \leq t \leq 1,
		\end{array}	
		\right. \nonumber \\[1em] 
		&& y_c(t) = \left\{
		\begin{array}{ll}\displaystyle
			0    & \displaystyle {\rm if}\ 0 \leq t \leq \frac{\|X_i\|}{\|X_i\| + \|Y_i\|}, \\[1em]
			\displaystyle \left( 1 - (1-t) \frac{\|Y_i\| + \|X_i\| }{\|Y_i\|} \right) y_{c}
			& \displaystyle {\rm if}\ \frac{\|X_i\|}{\|X_i\| + \|Y_i\| } \leq t \leq 1,
		\end{array}	
		\right. \nonumber \\[1em] 
		&& \hspace{5.5cm} (b \in X_i, c \in Y_i, i=1,2,\ldots,m), \label{eqn:formula}
	\end{eqnarray}
	and its length $d(P)$ is equal to $v({\cal A};x,y)$.
	\item[{\rm (2)}]
	Moreover any path-space geodesic in $K$ belongs to the path space 
	for some $(x,y)$-concave arch.
	\end{itemize}
\end{Thm}
Indeed, the projection $P|_{X_i \cup Y_i}$ of the path $P$ (\ref{eqn:formula}) is 
the unique geodesic between $x|_{X_i \cup Y_i}$ and $y|_{X_i \cup Y_i}$ 
in $K({\cal A},X_i,Y_i) = [0,1]^{X_i} \cup [0,1]^{Y_i}$ (union of two cubes).
Therefore, by the above argument and Lemma~\ref{lem:product}, 
for proving (1)	it suffices to verify that $P$ belongs to $K({\cal A})$, 
i.e., the nonzero support of 
$\sum_{b \in B} x_b(t) b + \sum_{c \in C} y_c(t) c$
for every $t \in [0,1]$ is a subset of a stable set in ${\cal A}$.
Notice from the concavity condition (\ref{eqn:concave}) that it holds that
\[
0 < \frac{\|X_1\|}{\|X_1\|+ \|Y_1\|} < \frac{\|X_2\|}{\|X_2\|+ \|Y_2\|} < \cdots < \frac{\|X_m\|}{\|X_m\|+ \|Y_m\|} < 1. 
\]
Accordingly the nonzero support of $P(t)$ changes as 
\[
B = U_0, U_0 \cap U_1, U_1,U_1 \cap U_2, U_2, \ldots,U_m=C.
\]

Proving (2) needs more effort.
In fact, the geodesic $P$ in $K({\cal A})$ for a nonconcave arch ${\cal A}$ belongs to 
a subspace $K({\cal A}') \subseteq K({\cal A})$ for 
a concave subarch ${\cal A}'$ of ${\cal A}$ 
(that corresponds to the extreme points of the convex hull of $(0,0)$ and $\xi_k$s).
In Appendix, 
we give a proof of this fact 
by simplifying the argument of Owen~\cite[Sections 4.1 and 4.2]{Owen11}.

Notice that the above argument 
does not use the CAT(0) property of $K(G)$.
In fact, 
without knowing the CAT(0) property of $K(G)$ (Theorem~\ref{thm:Gromov}) 
and Lemma~\ref{lem:CAT(0)=uniquegeodesic}, 
the unique geodesic property of $K(G)$
can directly be derived from the reduction techniques (R1-R4) below.

Consider points $x,y$ in $K(G)$ for general graph $G$. 
Then the situation reduces to the above special case 
by the following (R1) and (R2).
\begin{itemize}
	\item[(R1)] Let $B := \supp x$, $C := \supp y$.
	Consider the projection $z \mapsto z|_{B \cup C}$, 
	which is a strictly-nonexpansive retraction and fixes $x,y$.
	Hence any geodesic between $x,y$ must belong to 
	the subcomplex $K(\tilde G)$, 
	where $\tilde G$ is the subgraph of $G$ induced by $B \cup C$.
	Then $\tilde G$ is a bipartite graph 
	with color classes $B \setminus Z$, $C \setminus Z$, 
	where $Z \supseteq B \cap C$ is the set of isolated vertices in $\tilde G$.
	\item[(R2)] Hence we may assume from the first that $G$ is 
	such a bipartite graph.
	Then $K(G) = K(G') \times [0,1]^Z$, where $G'$ 
	is the subgraph of $G$ induced by non-isolated vertices.
	Then a geodesic in $K(G)$ is the product of geodesics 
	in $K(G')$ and in $[0,1]^Z$ (Lemma~\ref{lem:product}).
\end{itemize}
Thus the geodesic problem reduces to the above bipartite case.
In addition to Theorem~\ref{thm:owen}, 
to establish the unique geodesic property, 
two more properties are needed:
\begin{itemize}
	\item[(R3)] For points $x,y \in K(G)$ with $\supp x = B$, $\supp y = C$, 
	every geodesic connecting $x,y$ belongs to the path space $K({\cal A})$ for some arch ${\cal A}$.  
	\item[(R4)] There is a unique arch ${\cal A}^*$ 
	that attains the minimum $\min_{\cal A} v({\cal A};x,y)$. 
\end{itemize}
Let us outline the proof of (R3) and (R4); 
we will prove them in more general setting of modular semilattices. 
Suppose that a geodesic $P$ passes through cubes 
$[0,1]^{U_0}$, $[0,1]^{U_1},\ldots$.
Suppose that $U_{i+1} \cap B \not \subseteq U_{i} \cap B$. 
Consider the projection that sends coefficients of $(U_{i+1} \setminus  U_{i}) \cap B$ 
to zero, which is a strictly-nonexpansive retraction fixing $y$. 
Apply this map to $P$ from the moment when $P$ enters $[0,1]^{U_{i+1}}$, 
which yields an $(x,y)$-path with the length not greater than $d(P)$.
Hence $U_{i+1} \cap B \subseteq U_{i} \cap B$ necessarily holds, 
and consequently $(B = U_0, U_1,...)$ becomes an arch to establish (R3). 

For (R4), recalling Figure~\ref{fig:arch}, 
associate each stable set $U \in {\cal S}(G)$ 
with point $\xi(U) := (\sum_{b \in U \cap B} x_b^2, \sum_{c \in U \cap C}  y_c^2)$ in $\RR^2$.
Consider the convex hull $Q$ of all points $\xi(U)$ for $U \in {\cal S}(G)$, 
which is contained in square $[0, \|x\|^2] \times [0, \|y\|^2]$ 
and contains $(0,0), (\|x\|^2, 0), (0, \|y\|^2)$
as extreme points.
Then the sequence of stable sets mapped to nonzero extreme points
is the unique arch ${\cal A}^*$ that attains the minimum of $v({\cal A}; x,y)$.

\begin{Rem}\label{rem:MSSP}
	Now the geodesic can be constructed via arch ${\cal A}^*$, 
	which is algorithmically obtained as follows. 
	Indeed, ${\cal A}^*$ can be found by the linear optimization over $Q$ with objective vector 
	$((1- \lambda), \lambda))$ for parameter $\lambda \in [0,1]$. 
	This is equivalent to
	the following problem MSSP --- the maximum weight stable set problem in bipartite graph $G$ --- 
	with parameter $\lambda \in [0,1]$:
	\begin{eqnarray}
	{\rm MSSP}: \quad {\rm Max}. && (1- \lambda) \sum_{b \in B \cap U} x_b^2 + \lambda \sum_{c \in C \cap U} y_c^2 \nonumber \\
	{\rm s.t.} && U: \mbox{stable set of $G$}.  \nonumber
	\end{eqnarray}
	As did in~\cite{OwenProvan11} (see also \cite[Chapter 8]{BacakBook}), 
	MSSP reduces to  
	to the minimum cut problem in the network
	constructed from $G$ and $x,y$; see Figure~\ref{fig:network}.
	Then the cut $T$ having the minimum cut capacity
	corresponds to stable set $(T \cap B) \cup (C \setminus T)$ having the maximum weight. 
	Hence, via the max-flow min-cut theorem, 
	the arch is obtained by a (parametric) maximum flow algorithm. 
\end{Rem}

\begin{figure} 
	\begin{center} 
		\includegraphics[scale=0.8]{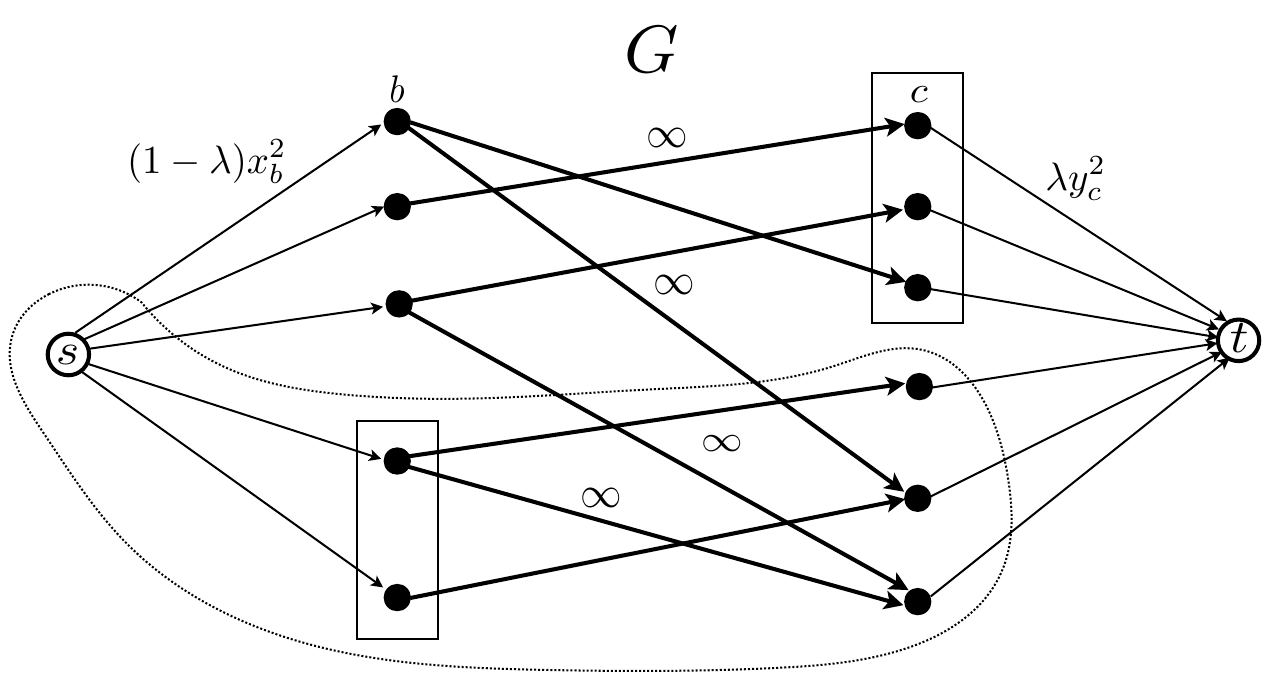}
		\caption{The network for MSSP: The cut surrounded by the dotted line has finite capacity 
			and corresponds to the stable set indicated by the square boxes.}  
		\label{fig:network}         
	\end{center}
\end{figure}

\subsection{Orthoscheme complex}\label{subsec:orthoscheme}

Here we formally introduce the orthoscheme complex of a graded poset.
For this purpose, we need to set up basic terminologies on posets. 
A {\em poset} ({\em partially order set}) is a set ${\cal P}$ 
endowed with a partial order relation $\preceq$ on ${\cal P}$, 
where $p \prec q$ means $p \preceq q$ and $p \neq q$.
A pair $p,q \in {\cal P}$ is said to be {\em comparable} if $p \preceq q$ or $q \preceq p$, 
and {\em incomparable} otherwise.
The {\em interval} $[p,q]$ of elements $p,q$ with $p \preceq q$ is the set of elements 
$u$ with $p \preceq u \preceq q$.
If $[p,q] = \{p,q\}$ and $p \neq q$, then we say that $q$ {\em covers} $p$, 
and $(p,q)$ is called a {\em covering pair}.
A {\em chain} of poset ${\cal P}$ is a pairwise comparable set 
of elements, which is often denoted by $p_0 \prec p_1 \prec \cdots$.
The {\em length} of a chain is defined as its cardinality minus one. 
A {\em grade function} of ${\cal P}$ is 
an integer-valued function $r:{\cal P} \to \ZZ$ 
such that $r(q) = r(p) + 1$ holds for all covering pairs $(p,q)$.
A poset ${\cal P}$ is called {\em graded} if it admits a grade function.
For $p \preceq q$, let $r[p,q] := r(q) - r(p)$, which is equal 
to the length of any maximal chain from $p$ to $q$. 
For posets ${\cal P}$ considered in this paper, we assume:
\begin{itemize}
\item[(F)] There is a finite number $N$ such that the length of every chain 
 is at most $N$.
\end{itemize}

Let ${\cal P}$ be a graded poset with grade function $r$.
The {\em simplex} of a chain $p_0 \prec p_1 \prec \cdots \prec p_n$
is the set of all formal convex combinations $x = \sum_{i=0}^n \lambda_i p_i$ 
of elements in the chain, where ``convex" means that  
the coefficients $\lambda_i$ satisfy
$\sum_{i=0}^n \lambda_i = 1$ and $\lambda_i \geq 0$ $(i=0,1,\ldots,n)$. 
Let $K({\cal P})$ denote the union of all simplices of chains in ${\cal P}$, or equivalently, the set of all formal convex 
combinations $x$ of elements in ${\cal P}$ 
such that  $\supp x$ is a chain of ${\cal P}$.
In other words, $K({\cal P})$ is 
a geometric realization of the order complex of ${\cal P}$.
Next we define a metric on $K({\cal P})$.
For a simplex $\varDelta$ of a chain 
$p_0 \prec p_1 \prec \cdots \prec p_n$, 
define map $\varphi_{\varDelta}: \varDelta \to \RR^{n}$ by
\begin{equation*}
\varphi_{\varDelta} (x) :=  \sum_{i=1}^n \lambda_i (e_1 + e_2 + \cdots + e_{r[p_0,p_i]}), 
\quad (x =\sum_{i=0}^n \lambda_i p_i \in K({\cal P})).
\end{equation*}
If $p_i$ covers $p_{i-1}$ for each $i$, then the above $\varphi_{\varDelta} (x)$ is also written as 
\begin{equation}\label{eqn:varphi_delta2} 
\varphi_{\varDelta} (x) =  \sum_{i=1}^{n} (\lambda_i + \cdots + \lambda_{n}) e_i.
\end{equation} 
For two points $x,y \in K({\cal P})$ belonging to 
a common simplex $\varDelta$, define distance $d(x,y)$ by the $l_2$-distance in the image 
$\varphi_{\varDelta}(\varDelta)$:
\begin{equation}\label{eqn:d_Delta}
d(x, y) := 
\| \varphi_{\varDelta}(x) - \varphi_{\varDelta}(y)  \| \quad (x,y \in \varDelta).
\end{equation}
Namely, $\varphi_{\varDelta}$ maps a chain 
to vertices of the orthoscheme (\ref{eqn:orthoscheme}).
Accordingly, points in the simplex 
and their distance are realized in the orthoscheme in Euclidean space $\RR^n$.
Note that the distance (\ref{eqn:d_Delta}) 
does not depend on the choice of a common simplex.
Also the neighborhood of each point is determined, which  
generates a topology on $K({\cal P})$.
The length $d(P)$ of a path $P$ in $K({\cal P})$
is defined by (\ref{eqn:length}), 
where $t_i$s are taken so that 
$P(t_i),P(t_{i+1})$ belong to a common simplex
and their distance is measured by (\ref{eqn:d_Delta}).
For the distance $d(x,y)$ of arbitrary points $x,y \in K({\cal P})$
is defined as the infimum of $d(P)$ for all $(x,y)$-paths.
The resulting metric simplicial complex $K({\mathcal P})$
is called the {\em orthoscheme complex} of ${\mathcal P}$.
By the assumption~(F),  
$K({\cal P})$ is an $M_0$-simplicial complex 
with finite isometry types of simplices, and is a complete geodesic space.

Let ${\cal P}, {\cal P}'$ be graded posets.
A map $\phi:{\cal P} \to {\cal P}'$ 
is called {\em order-preserving} if 
$\phi(p) \preceq \phi(q)$ holds for all $p,q \in {\cal P}$ with $p \preceq q$.
An order-preserving map $\phi:{\cal P} \to {\cal P}'$ maps a chain in ${\cal P}$ to a chain in ${\cal P}'$, 
and hence is extended 
to a map $\phi:K({\cal P}) \to K({\cal P}')$ in a natural way:
\[
\phi(x) := \sum_{i} \lambda_i \phi(p_i) \quad (x = \sum_{i} \lambda_i p_i \in K({\cal P})).
\] 
An order-preserving map $\phi:{\cal P} \to {\cal P}'$
is called {\em nonexpansive} if for every covering pair $(p,q)$ in ${\cal P}$, 
$(\phi(p), \phi(q))$ is a covering pair or $\phi(p) = \phi(q)$.
\begin{Lem}\label{lem:strictly_decrease}
	Let $\phi:{\cal P} \to {\cal P}'$ be a nonexpansive order-preserving map.
	\begin{itemize}
		\item[{\rm (1)}] The extension $\phi:K({\cal P}) \to K({\cal P}')$ is nonexpansive (and continuous).
		\item[{\rm (2)}] 
		For points $x = \sum_{i=0}^n \lambda_i p_i$, $x' = \sum_{i=0}^n \lambda'_i p_i \in K({\cal P})$ in a common simplex,
		if $\phi(p_i) = \phi(p_{i-1})$ and $\lambda_i + \lambda_{i+1} + \cdots + \lambda_n \neq \lambda'_i + \lambda'_{i+1} + \cdots + \lambda'_n$, 
		then it holds 
		\[
		d(\phi(x), \phi(x')) < d(x,x').
		\]
	\end{itemize}
\end{Lem}
\begin{proof}
	(1). Take arbitrary 
	two points $x = \sum_{i=0}^n \lambda_i p_i$ and $x' = \sum_{i=0}^n \lambda_i' p_i$ 
	in a common simplex in $K({\cal P})$.
	We can assume that each $(p_i,p_{i+1})$ is a covering pair.
	It suffices to show $d(\phi(x),\phi(x')) \leq d(x,x')$.
	Let $I$ denote the set of indices $i (> 0)$ with $\phi(p_i) \succ \phi(p_{i-1})$ 
	(i.e., $\phi(p_i)$ covers $\phi(p_{i-1})$).
	Suppose that $I = \{i_1,i_2,\ldots,i_k\}$ for $i_1 < i_2 < \cdots < i_k$.
	Then we have
	$\phi(x) = \sum_{l=0}^k (\lambda_{i_l} + \lambda_{i_l+1} + \cdots + \lambda_{i_{l+1} -1}) \phi(p_{i_l})$ and
	$\phi(x') = \sum_{l=0}^k (\lambda'_{i_l} + \lambda'_{i_l+1} + \cdots + \lambda'_{i_{l+1} -1}) \phi(p_{i_l})$, 
	where $i_0 := 0$.
	By (\ref{eqn:varphi_delta2}), we have
	\begin{eqnarray}
	d(\phi(x),\phi(x'))^2 & = &\sum_{i \in I}(\lambda_{i} + \lambda_{i+1} + \cdots + \lambda_n 
	- \lambda'_{i} - \lambda'_{i+1} - \cdots - \lambda'_n)^2 \label{eqn:d(phi(x),phi(x')} \\
	& \leq & \sum_{i=1}^n(\lambda_{i} + \lambda_{i+1} + \cdots + \lambda_n 
	- \lambda'_{i} - \lambda'_{i+1} - \cdots - \lambda'_n)^2 \nonumber \\
	& =&  d(x,y)^2. \nonumber
	\end{eqnarray}

	(2). In the above inequality, 
	the index $i$ does not belong to $I$,  has nonzero term, 
	and the inequality holds strictly. 
\end{proof}

To proceed the argument, we need further notation 
on lattice and semilattice.
The {\em join} and {\em meet} of two elements $p,q$ in a poset ${\cal P}$ are 
the unique minimum common upper bound and
 the unique maximum common lower bound, respectively, of $p,q$.
 The join and meet of $p,q$ (if they exist) are denoted by $p \vee q$ and $p \wedge q$, respectively. 
A {\em lattice} is a poset in which every pair of elements has both join and meet. 
A {\em (meet-)semilattice} is a poset in which every pair of elements has meet.
The minimum element in a semilattice is denoted by $0$.
We only consider semilattices that are graded, where
the grade of the minimum element $0$ is supposed to be $0$, and 
the grade of an element $p$ is also called the {\em rank} of~$p$.
The maximum rank of an element is 
called the {\em rank} of the semilattice, which is finite by (F).  
In a semilattice,
two elements $p,q$ are said to be {\em bounded} if they have a common upper bound.
Notice that $p$ and $q$ are bounded if and only if the join $p \vee q$ exists, 
which is given by the meet of all common upper bounds of $p,q$. 

An {\em ideal} in a poset ${\cal P}$ is a subset $S$ such 
that $p \preceq q \in S$ implies $p \in S$.
For an element $a$ of a poset ${\cal P}$, 
the {\em principal ideal} ${\cal I}(a)$ of $a$
is defined as the set of all elements $p$ with $p \preceq a$.
Dually the {\em principal filter} ${\cal F}(a)$ of $a$
is defined as the set of all elements $p$ with $p \succeq a$.
A {\em subsemilattice} of a semilattice ${\cal L}$ is 
a subset that is closed under meet. 
A {\em sublattice} is a subset that  is closed under meet and join.
\begin{Ex}\label{ex:distributivelattice}
	A {\em distributive lattice} is a lattice ${\cal D}$ 
	satisfying distributive law 
	$p \wedge (q \vee q') = (p \wedge q) \vee (p \wedge q')$ 
	and $p \vee (q \wedge q') = (p \vee q) \wedge (p \vee q')$ 
	for every triple $p,q,q' \in {\cal L}$.
	The family of ideals in a (finite) poset ${\cal P}$ is a distributive lattice,  
	where the partial order relation on ideals is defined by inclusion order; 
	then $\wedge = \cap$ and $\vee = \cup$.  
	Birkhoff representation theorem says that
	any distributive lattice is always obtained in this way; see \cite[Chapter V]{Birkhoff} and \cite[Chapter II]{Gratzer}.
	
	Suppose that a distributive lattice ${\cal D}$ is represented by a poset ${\cal P}$. 
	Then  \cite[Proposition 7.11]{CCHO14} shows that the orthoscheme complex $K({\cal D})$ is isomorphic to the convex polytope
	\begin{equation}\label{eqn:orderpolytope}
	\{ x \in [0,1]^{\cal P}  \mid x_v \geq x_u\ (u,v \in {\cal P}: u \preceq v) \}
	\end{equation}
	in Euclidean space $\RR^{\cal P}$, which is known as the {\em order polytope} of ${\cal P}$. 
\end{Ex}

\begin{Ex}\label{ex:Booleansemilattice}
	A {\em Boolean lattice} is a distributive lattice such 
	that every element $p$ has an element $q$, called a {\em complement} of $p$,  
	such that $p \wedge q = 0$ and $p \vee q = 1$ (the maximum element). 
	A Boolean lattice here is a lattice $2^{V}$ of all subsets of 
	a finite set $V$.
	By (\ref{eqn:orderpolytope}) with regarding $V$ as a poset with no relation, 
	the orthoscheme complex $K(2^V)$ 
	is isometric to Euclidean cube $[0,1]^V$.
	Consequently, the rooted cubical complex $K({\cal S})$ is also the orthoscheme complex $K({\cal S})$, 
	where the abstract simplicial complex ${\cal S}$ is regarded 
	as a (graded) poset ordered by inclusion.
	The poset of an abstract simplicial complex 
	 is identified with a semilattice ${\cal B}$ such that 
	every principal ideal is a Boolean lattice.
	Indeed, consider the set $V$ of all rank-1 elements of ${\cal B}$, and 
	consider the abstract simplicial complex ${\cal S}$ on $V$ consisting 
	of all subsets $S$ with $\bigvee S \in {\cal B}$. Then ${\cal S}$ is isomorphic to ${\cal B}$.

	A flag simplicial complex is equivalent to 
	a {\em Boolean semilattice}, 
	which is defined as a semilattice ${\cal B}$
	such that every principal ideal of ${\cal B}$ is a Boolean lattice 
	and ${\cal B}$ satisfies the following lattice-theoretic flag condition:
	\begin{itemize}
		\item[(LFL)] For every triple $u,v,w$ of elements, 
		their join $u \vee v \vee w$ exists if and only if
	 all of $u \vee v$, $v \vee w$, $w \vee u$ exist.
	\end{itemize}
	Indeed, in the above construction of ${\cal S}$, 
	(LFL) is rephrased as: For $S,T,U \in {\cal S}$, $S \cup T \cup U \in {\cal S}$ if and only 
	if $S \cup T,T \cup U, U \cup S \in {\cal S}$.
	It is easy to see (by induction) that this is equivalent to (FL). 
	Thus $K({\cal B})$ is isometric to CAT(0) rooted cubical complex $K(G)$. 
\end{Ex}
We see in the next subsection
a common generalization of a distributive lattice and Boolean semilattice.

Let ${\cal P}$ be a graded poset.
Even if a subset ${\cal P}' \subseteq {\cal P}$ 
becomes a graded poset by the restriction of $\preceq$, 
the orthoscheme complex $K({\cal P}')$, which is a subset of $K({\cal P})$,  
is not necessarily a subspace of $K({\cal P})$, 
since their metrizations may be different. 
A necessary and sufficient condition
for $K({\cal P}')$ to be a subspace of $K({\cal P})$ 
is the following rank-preserving property: 
\begin{itemize}
\item[(RP)] Any covering pair of ${\cal P}'$ is a covering pair of ${\cal P}$.	
\end{itemize}
Then the shape of each simplex in $K({\cal P}')$ 
is the same in $K({\cal P})$, and $K({\cal P}')$ is viewed as a subspace of $K({\cal P})$.
Examples of such subsets include intervals, principal ideals, and filters. 

Consider maps $p \mapsto a \wedge p$ and $p \mapsto a \vee p$, 
when they are defined for all $p$.
Then they are obviously order-preserving, and extended to 
$K({\cal P}) \to K({\cal I}(a))$ and $K({\cal P}) \to K({\cal I}(a))$. 
We are interested in the situation where they are nonexpansive (retractions).
An element $a$ in ${\cal L}$ is called {\em modular} 
if $a$ has join $x \vee a$ and meet $x \wedge a$ with every element $x$, and satisfy 
\begin{equation}
r(a) + r(x) = r(a \vee x) + r(a \wedge x).
\end{equation}
\begin{Lem}\label{lem:decomposition}
	Let $a \in {\cal P}$ be a modular element.
	\begin{itemize}
		\item[{\rm (1)}] Maps $a \wedge$ and $a \vee$ are order-preserving nonexpansive retractions to ${\cal I}(a)$ and to ${\cal F}(a)$, respectively.
		\item[{\rm (2)}] Subspaces $K({\cal I}(a))$ and $K({\cal F}(a))$ are strictly-convex.
		\item[{\rm (3)}] For a path $P$ in $K({\cal P})$, it holds
		\begin{equation}\label{eqn:decomposition}
		d(P)^2 \geq d(a \vee P)^2 + d( a \wedge P)^2.
		\end{equation}
		If both $a \vee P$ and $a \wedge P$ are geodesics, 
		then the equality holds in (\ref{eqn:decomposition}) and $P$ is a geodesic.
	\end{itemize}
\end{Lem}
\begin{proof}
	(1). For a covering pair $(p,p')$, by modularity of $a$, exactly one of the following holds:
	\begin{itemize}
		\item $a \wedge p'$ covers $a \wedge p$ and $a \vee p' = a \vee p$.
		\item $a \wedge p' = a \wedge p$ and $a \vee p'$ covers $a \vee p$.
	\end{itemize}
	This follows from $1 = r[a \vee p, a \vee p'] + r[a \wedge p, a \wedge p']$, 
	which is obtained by subtracting $r(a) + r(p) = r(a \vee p) + r(a \wedge p)$ from 
	$r(a) + r(p') = r(a \vee p') + r(a \wedge p')$. 
	In particular, both $a \wedge$ and $a \vee$ 
	are nonexpansive retractions.
	
	(3). To show (\ref{eqn:decomposition}), it suffices to show $
	d(x,y)^2 = d(a \wedge x, a \wedge y)^2 + d(a \vee x, a \vee y)^2
	$
	for points $x,y$ in a common simplex.
	Indeed, the argument in the proof of Lemma~\ref{lem:product} 
	is applicable in a straightforward way, since
    $d(P(t_{i-1}),P(t_{i}))^2 = d(a \wedge P(t_{i-1}), a \wedge P(t_{i}))^2 + d(a \vee P(t_{i-1}), a \vee P(t_{i}))^2$ holds for a sufficiently fine subdivision $0 = t_0 < t_1 < \cdots < t_N = 1$.

    We can assume that $x = \sum_{i} \lambda_i p_i$ and $y =\sum_{i} \mu_i p_i$ 
    for a maximal chain $p_0 \prec p_1 \prec \cdots \prec p_n$.
    We use the method of the proof of Lemma~\ref{lem:strictly_decrease}~(1).
    Let $I$ be the set of indices $i \in \{1,2,\ldots,n\}$ 
    such that $a \wedge p_i$ covers $a \wedge p_{i-1}$.
    Then, by the argument in (1), $J = \{1,2,\ldots,n\} \setminus I$ is the set of indices 
    $i$ such that $a \vee p_i$ covers $a \vee p_{i-1}$.
    Therefore, by (\ref{eqn:d(phi(x),phi(x')}), we have
    \[
    d(x,y)^2  =  \sum_{i \in I \cup J} (\lambda_i + \cdots + \lambda_n - \mu_i - \cdots - \mu_n)^2 
     =  d(a \wedge x, a \wedge y)^2 + d(a \vee x, a \vee y)^2.  
    \]
   
    Suppose that both $a \vee P$ and $a \wedge P$ are geodesics.
    For $0 \leq t < t' \leq 1$, 
    choose any sufficiently fine subdivision $t= t_0 < t_1 < \cdots < t_m =t'$. 
    Then we have
    \begin{eqnarray*}
    && d(P(t),P(t')) \leq  \sum_{i=1}^m d(P(t_{i-1}),P(t_{i}))  \\
    && =   \sum_{i=1}^m \sqrt{d(a \wedge P(t_{i-1}), a \wedge P(t_i))^2 +d(a \vee P(t_{i-1}), a \vee P(t_i))^2 } \\
    && =  \sum_{i=1}^m (t_{i}- t_{i-1}) \sqrt{d(a \wedge P(0), a \wedge P(1))^2 +d(a \vee P(0), a \vee P(1))^2} \\
    && =  |t' -t| \sqrt{d(a \wedge P(0), a \wedge P(1))^2 +d(a \vee P(0), a \vee P(1))^2} \\
    && =  \sqrt{d(a \wedge P(t), a \wedge P(t'))^2 +d(a \vee P(t), a \vee P(t'))^2} \\
    && \leq  d(P(t),P(t')),	
    \end{eqnarray*}
    where the last inequality follows from the established (\ref{eqn:decomposition}).
   From this, we see that the equality holds in (\ref{eqn:decomposition}) and $P$ is a geodesic.

	(2). Let $u,v \in K({\cal I}(a))$.
	Take a polygonal path $P$ connecting $u$ and $v$.
	Then the image $a \wedge P$ is a (polygonal) path in $K({\cal I}(a))$ 
	connecting $u = a \wedge u$ and $v = a \wedge v$.	
	Suppose that $P$ meets $K({\cal P}) \setminus K({\cal I}(a))$.
	We show $d( a \wedge P) < d(P)$.
	We can take two points $x,y$ in $P$ 
	such that the segment $[x,y]$, which is a part of $P$, belongs to a common simplex, 
	$x \in K({\cal I}(a))$, and $y \not \in K({\cal I}(a))$.
	Suppose that $x = \sum_{i=0}^n \lambda_i p_i$ and $y = \sum_{i=0}^n \lambda'_i p_i$.
	For some index $k$, it necessarily holds that $p_j \preceq a$ $(j \leq k)$, 
	$p_{j} \not \preceq a$ $(j > k)$, $\sum_{j: j> k} \lambda_i = 0$, 
	and $\sum_{j: j> k} \lambda'_i \neq 0$.
	Then $p_k = a \wedge p_k = a \wedge p_{k+1}$ must hold (provided $p_{k+1}$ covers $p_k$).
	By Lemma~\ref{lem:strictly_decrease}~(2), 
	we have $d(a \wedge x, a \wedge y) < d(x,y)$.
	Consequently $d(a \wedge P) < d(P)$.
	Thus every shortest path between $x$ and $y$ belongs to $K({\cal I}(a))$.
	
	For $K({\cal F}(a))$, reverse the partial order of ${\cal P}$ 
	and consider the corresponding orthoscheme complex, 
	which is isometric to the original $K({\cal P})$. 
	Then we obtain the statement for $K({\cal F}(a))$.
\end{proof}
A {\em modular lattice} is a graded poset (lattice) such that every element 
is a modular element; this may be an unusual definition of a modular lattice 
but is equivalent to the standard one; see~\cite[Sections 50--52]{Birkhoff}.
In a modular lattice, 
we can use the above lemma freely. 
Also the above proof is applied to show the following:
\begin{Lem}\label{lem:decomposition2}
Let ${\cal L}$ be a semilattice such that every principle ideal is a modular lattice.
For $a \in {\cal L}$, the map $a \wedge$ is an order-preserving nonexpansive retraction to ${\cal I}(a)$, 
and $K({\cal I}(a))$ is a strictly-convex subspace of $K({\cal L})$.
\end{Lem}

\subsection{Median orthoscheme complex}\label{subsec:median}
A {\em median semilattice} is a semilattice ${\cal D}$ 
such that every principal ideal of ${\cal D}$
is a distributive lattice and ${\cal D}$ 
satisfies the lattice-theoretic flag condition (LFL).
By the definition, a median semilattice is 
a common generalization of 
a distributive lattice and Boolean semilattice.
The former is represented by the family of ideal 
in a poset (Example~\ref{ex:distributivelattice}) 
and the latter is the family of all stable sets of a graph (Example~\ref{ex:Booleansemilattice}).
A median semilattice admits a common generalization of these representations, from which an explicit description of its orthoscheme complex is given.

A {\em PIP (Poset with Inconsistent Pairs)} $G_{\preceq} = (V,E,\preceq)$ is a pair of an undirected 
graph $G = (V,E)$ and  
a partial order relation $\preceq$ on vertex set $V$ 
such that $uv \in E$ and $u \preceq u'$ imply $u'v \in E$.
This concept appeared in \cite{barthlemy1993median}; 
the name PIP is due to \cite{ArdilaOwenSullivant2012}.
A {\em stable ideal} (or {\em consistent ideal}) is a vertex subset 
such that it is a stable set relative to the graph $G$ 
and an ideal relative to the poset $(V,\preceq)$.
Let ${\cal S}(G_{\preceq}) \subseteq 2^V$ be the poset of all stable ideals 
ordered by inclusion. 
Notice that ${\cal S}(G_{\preceq})$ is not an abstract simplicial complex.
\begin{Prop}[{\cite{barthlemy1993median}}]
	For a PIP $G_{\preceq}$, the family ${\cal S}(G_{\preceq})$ of stable ideals is 
	a median semilattice with $\wedge = \cap$.
	Conversely, every median semilattice ${\cal D}$ is isomorphic to 
	${\cal S}(G_{\preceq})$ for some PIP $G_{\preceq}$.
\end{Prop}
The construction of such a PIP $G_{\preceq}$ is as follows.
The vertex set $V$ of $G_{\preceq}$ is the set of all join-irreducible elements of ${\cal D}$, where
a {\em join-irreducible} element 
is an element that is not $0$ and cannot be represented as the join of other elements.
The partial order $\preceq$ on $V$ is the restriction of the partial order of ${\cal D}$.
A pair $(u,v)$ of vertices 
has an edge in $G =(V,E)$ if and only if the join of $u,v$ does not exist. 
The resulting $G_{\preceq} = (V,E,\preceq)$ is actually a PIP, and
${\cal D}$ is isomorphic to ${\cal S}(G_{\preceq})$, 
where an isomorphism is given by $p \mapsto V \cap {\cal I}(p)$.
In particular, median semilattice ${\cal D}$ is embedded to Boolean semilattice ${\cal S}(G)$.
This Boolean semilattice ${\cal S}(G)$ is called 
the {\em Boolean extension} of ${\cal D}$, and is also denoted by $\overline{\cal D}$.

Consider the orthoscheme complex $K({\cal D})$ 
of a median semilattice ${\cal D}$, 
which is called a {\em median orthoscheme complex}.
The next proposition shows that 
that $K({\cal D})$ is realized as a CAT(0) subspace in CAT(0) 
rooted cubical complex $K(\overline{\cal D}) = K(G)$.
\begin{Prop}[{\cite[Sectioin 7.6]{CCHO14}}]\label{prop:median_semilattice}
	Let ${\cal D}$ be a median semilattice.
	\begin{itemize}
		\item[{\rm (1)}] The median orthoscheme complex $K({\cal D})$ is CAT(0).
		\item[{\rm (2)}] Suppose that ${\cal D}$ is represented by PIP 
		$G_{\preceq} = (V,E, \preceq)$. Then $K({\cal D})$
		is isometric to the subspace $K(G_{\preceq})$ of $K(G)$:
		\begin{equation*}
		K(G_{\preceq}) := \{ x \in K(G) \mid x_v \geq x_{v'}\ (v,v' \in V: v \preceq v') \},
		\end{equation*}
		where the isometry is given by 
		\begin{equation}\label{eqn:b-coordinate}
		x = \sum_{i}\lambda_i p_i \mapsto \sum_{v \in V} (\sum_{i: v \preceq p_i} \lambda_i) v.
		\end{equation}		
	\end{itemize}
\end{Prop}

Note that one can also associate PIP $G_{\preceq}$ 
with CAT(0) cubical complex $C(G_{\preceq})$, as in~\cite{ArdilaOwenSullivant2012}, 
which is different from $K(G_{\preceq})$.

The image of $x \in K({\cal D})$ by the isometry in (\ref{eqn:b-coordinate}) is called 
the {\em b-coordinate} of $x$, 
where ``b'' stands for Birkhoff. 
We write $x =_{\rm b} \sum_{v \in V}x_v v$ if the image of $x$ 
by the isometry (\ref{eqn:b-coordinate}) 
is $\sum_{v \in V} x_v v$.
Observe from (\ref{eqn:b-coordinate}) that
the meet and join maps work as projections as follows:
\begin{Lem}\label{lem:b-coordinate}
	Let $x =_{\rm b} \sum_{v \in V} x_v v \in K({\cal D})$. For $a \in {\cal D}$ with 
	$A = \{ v \in V \mid v \preceq a\}$, it holds 
	$x \wedge a =_{\rm b} \sum_{v \in A} x_v v = x|_{A}$ and 
	$x \vee a =_{\rm b} \sum_{v \in A} v + x|_{V \setminus A}$. 
\end{Lem}

Next we discuss geodesics in $K({\cal D})$, and show that Owen's formula is naturally extended. 
A {\em bipartite} PIP is a PIP $G_{\preceq}$ 
such that $G$ is a bipartite graph with color classes $B,C$ and has no isolated vertices,  
and any $b \in B$ and $c \in C$ are incomparable in $\preceq$.
Suppose that $G_{\preceq}$ is a bipartite PIP with color classes $B,C$.
An {\em arch} is a sequence $(B = U_0,U_1,\ldots,U_m = C)$ 
of stable ideals satisfying (\ref{eqn:arch}).
The path space $K({\cal A})$ is defined as $K(G_{\preceq}) \cap (\bigcup_i [0,1]^{U_i})$.

Let $x,y \in K(G_{\preceq})$ with $\supp x=B$ and $\supp y =C$.
For an arch ${\cal A} = (B = U_0,U_1,\ldots,U_m = C)$, 
$X_i, Y_i, \|X_i\|, \|Y_i\|$ are defined 
by (\ref{eqn:X_i}) (\ref{eqn:Y_i}), (\ref{eqn:|X_i|}), and (\ref{eqn:|Y_i|}).
Also $v({\cal A};x,y)$ is defined by (\ref{eqn:v(A;x,y)}).
An $(x,y)$-concave arch is an arch satisfying (\ref{eqn:concave}). 
In this setting, precisely the same statement of Theorem~\ref{thm:owen} holds.
\begin{Prop}\label{prop:formula_for_mediansemilattice}
	\begin{itemize}
		\item[{\rm (1)}]
	For an $(x,y)$-concave arch ${\cal A} = (B = U_0,U_1,\ldots,U_m =C)$, 
	the unique geodesic connecting $x,y$ in $K({\cal A})$ 
	is given by (\ref{eqn:formula}), and its length is equal to $v({\cal A};x,y)$.
\item[{\rm (2)}]
	Moreover, any path-space geodesic belongs 
	to the path space for some $(x,y)$-concave arch. 
\end{itemize} 
\end{Prop}
\begin{proof}
	(1). $K(G_{\preceq})$ is a subspace of the cubical complex $K(G)$.
	An arch ${\cal A}$ for $G_{\preceq}$ is an arch for $G$.
	The path space $K({\cal A})$ for $K(G_{\preceq})$ is a subspace of the (cubical) path space for $K(G)$.
	Therefore it suffices to show that the path $P$ defined by (\ref{eqn:formula}) 
	is actually a path in $K({\cal A})$, i.e.,  
	that for all $t \in [0,1]$, it holds $x_v(t) \geq x_{v'}(t)$ if $v \preceq v'$. 	
    Consider $v,v' \in V$ with $v \preceq v'$.
    Then $v,v' \in B$ or $v,v' \in C$. Suppose that $v,v' \in B$. 
    Since $x \in K(G_{\preceq})$ it holds $x_v(0) = x_v \geq x_{v'} = x_{v'}(0)$.
    Any stable ideal containing $v'$ must contain $v$. 
	Consequently, if $v \in X_{i}$ and $v' \in X_{i'}$, then $i' \leq i$ 
	and hence $\|X_{i'}\|/(\|X_{i'}\| + \|Y_{i'}\|) \leq \|X_{i}\|/(\|X_{i}\| + \|Y_{i}\|)$.
	Thus $x_v(t) =  (1- t(\|X_{i}\| + \|Y_{i}\|)/\|X_{i}\|)x_v \geq  (1- t(\|X_{i'}\| + \|Y_{i'}\|)/\|X_{i'}\|) x_{v'} = x_{v'}(t)$, as required.
	The case of $v,v' \in C$ is shown in a similar way. 
	
	(2). If ${\cal A}$ is nonconcave, 
	then it is also nonconcave for $K(G)$, 
	and the path-space geodesic $P$ for the (cubical) path space in $K(G)$
	belongs to the (cubical) path space for some concave subarch ${\cal A}'$; see Appendix.
	Since this arch ${\cal A}'$ is also a concave arch for $K(G_{\preceq})$, 
	by (1), $P$ belongs to $K({\cal A}')$.
\end{proof}

The unique geodesic property for $K({\cal D})$ can also be established by proving 
the analogues of (R1-R4) in Section~\ref{subsec:cubical}.
In (R1), the projection $x \mapsto x|_{B \cup C}$ is 
also a well-defined strictly-nonexpansive retraction.
Indeed, if $\supp z$ is a stable ideal, 
then so is $\supp z|_{B \cup C} = (B \cup C) \cap \supp z$ (since $B \cup C$ is an ideal). 
Hence geodesics belongs to the strictly-convex subspace
corresponding to the PIP obtained by restricting $G_{\preceq}$ to $B \cup C$.
This PIP is a {\em semi-bipartite} PIP 
(with tri-partition $\{B \setminus Z, C \setminus Z, Z (\supseteq B \cap C)\}$) in the following sense.
A PIP $G_{\preceq} = (V,E, \preceq)$ is called semi-bipartite if it admits 
a tri-partition $\{ B', C', Z\}$ of $V$
such that the restriction $G'_{\preceq}$ of $G_{\preceq}$ to $B'\cup C'$ 
is a bipartite PIP with color classes $B',C'$, 
$Z$ is the set of isolated vertices of $G$, and
there are no $p \in B' \cup C'$ and $q \in Z$ with $p \preceq q$. 
Let $G_{\preceq}^0$ denote the restriction of $G_{\preceq}$ to $Z$, 
which has no edge and is merely a poset.
Then $K(G_{\preceq}) \subseteq K(G'_{\preceq}) \times K(G^0_{\preceq})$.
In contrast to the cubical case, 
the strict inclusion possibly holds. 
Fortunately the unique geodesic can be obtained as the product of 
those for $K(G'_{\preceq})$ and $K(G^0_{\preceq})$.
\begin{Lem}\label{lem:product_median}
	Let $G_{\preceq}$ be a semi-bipartite PIP with tri-partition $B',C',Z$.
	Let $x,y \in K(G_{\preceq})$ with $B' \subseteq \supp x \subseteq B' \cup Z$ and 
	$C' \subseteq \supp y \subseteq C' \cup Z$.
	Let $Q$ be a shortest path-space geodesic in $K(G'_{\preceq})$ 
	connecting $x|_{B' \cup C'}$ and $y|_{B' \cup C'}$, 
	and let $R$ be the geodesic in $K(G^0_{\preceq})$ connecting
	$x|_{Z}$ and $y|_{Z}$.
	Then the product $P = (Q,R)$ belongs to $K(G_{\preceq})$.
\end{Lem}
%
\begin{proof}
	Notice that $G^0_{\preceq}$ has no edge, and $K(G^0_{\preceq})$ is a convex polytope in $[0,1]^Z$ (Example~\ref{ex:distributivelattice}).
	Therefore the unique geodesic in $K(G^0_{\preceq})$ connecting $x|_{Z}$ and $y|_{Z}$ is 
	given by $t \mapsto \sum_{v \in Z} (1-t) x_v v + t y_v v$.
	Therefore it suffices to show that if $v \preceq v'$ for $v \in Z$ and $v' \in B' $, 
	then $x_v(t) \geq x_{v'}(t)$, where $x_{v'}(t)$ obeys (\ref{eqn:formula}) 
	for some $(x|_{B'}, y|_{C'})$-concave arch.
	Then $x_v(t) =  (1-t)x_v +  t y_v \geq (1-t)x_v \geq  (1- t(\|X_{i'}\| + \|Y_{i'}\|)/\|X_{i'}\|)) x_{v'} = x_{v'}(t)$, as required.
\end{proof}
Then the unique geodesic property of $K(G_{\preceq})$ can be shown by establishing
(R3) and (R4) in a similar way.
\begin{Rem}\label{rem:MSIP}
	Again the geodesic can be constructed via 
	${\cal A}^*$, which is also obtained by a network flow technique as in Remark~\ref{rem:MSSP}.
	In MSSP,
	replace ``stable set in $G$"  by ``stable ideal in $G_{\preceq}$."
	The arch ${\cal A}^*$ is obtained by solving
	the resulting problem MSIP. 
	Consider the network in Figure~\ref{fig:network}. 
	For $u,v \in B$ (resp. $C$), 
	add edge $vu$ with infinite capacity if $u \preceq v$ (resp. $v \preceq u$).
    Again cuts having finite capacity and stable ideals are 
    in one-to-one correspondence by $T \leftrightarrow (T \cap B) \cup (C \setminus T)$.  
	Thus MSIP is solved by a (parametric) max-flow algorithm. 
	Note that MSIP is equivalent to the problem 
	known as the {\em minimum weight ideal problem} in a poset, 
	where this reduction is classically known~\cite{Picard76}; 
	see also \cite[Section 7.1 (b)]{FujiBook}.
\end{Rem}

\section{Modular semilattice}\label{sec:modularsemilattice}
A {\em modular semilattice}~\cite{BVV93} is a semilattice ${\cal L}$ 
such that every principal ideal of ${\cal L}$ is a modular lattice 
and ${\cal L}$ satisfies the lattice-theoretic flag condition (LFL).
In this section, we deal with the orthoscheme complex of a modular semilattice.
The goal in this section is to prove the following, 
which implies the main theorem (Theorem~\ref{thm:main}) via Lemma~\ref{lem:CAT(0)=uniquegeodesic}.
\begin{Thm}\label{thm:unique_geodesic}
	Let ${\cal L}$ be a modular semilattice.
	Then the orthoscheme complex $K({\cal L})$ is uniquely geodesic.
\end{Thm} 
The proof is largely based on the idea 
mentioned in Sections~\ref{subsec:cubical} and \ref{subsec:median}, i.e., 
the formula (\ref{eqn:formula}) of path-space geodesics and 
the reduction techniques (R1-R4). The outline is as follows:
\begin{itemize}
	\item[(P0)] Let $x, y \in K({\cal L})$.
	Our goal is to give an explicit construction  
	of a geodesic between $x,y$ and to show its uniqueness.
	 Let $p$ and $q$ 
	 be the maximum elements in the nonzero supports of $x$ and $y$, respectively.
	  
	 \item[(P1)] Consider first an easier case where $p$ and $q$ have join.  
	 Then any geodesic between $x,y$ belongs 
	 to a strictly-convex subspace $K({\cal I}(p \vee q))$ (Lemma~\ref{lem:decomposition}).
	 Now ${\cal I}(p \vee q)$ is a modular lattice, and $K({\cal I}(p \vee q))$ is CAT(0) (Theorem~\ref{thm:modularlattice}).
	 The geodesic in $K({\cal I}(p \vee q))$ is unique.
	 Moreover the geodesic is easily constructed 
	 by taking a distributive sublattice ${\cal D}$ for which $K({\cal D})$ contains $x$ and $y$ 
	 (Lemmas~\ref{lem:apartment} and \ref{lem:modularlattice}).
	 \item[(P2)] Next consider the essential case 
	 where there is no element $u$ other than $0$ 
	 that has join with both $p$ and $q$.
	 In this case, an analogue of (R1) holds: 
	 \begin{itemize}
	 	\item Every geodesic connecting $x,y$ belongs to 
	 	the subspace $K(I(p,q))$ induced by modular subsemilattice 
	 	\[
	 	I(p,q) = \{ u \vee v \mid u \in {\cal I}(p), v \in {\cal I}(q) \}.
	 	\]
	 \end{itemize}
	 \item[(P3)] This modular semilattice $I(p,q)$ generalizes and plays roles of
	 a median semilattice represented by a bipartite PIP.
	 The concepts of an arch, path space, $(x,y)$-concavity, and $v({\cal A};x,y)$
	 are naturally generalized. 
	 By taking a special median subsemilattice (called a {\em distributive frame}) of $I(p,q)$, 
	 we obtain the formula of path-space geodesics
	 (Proposition~\ref{prop:formula_for_modularsemilattice}). 
	 We then establish (R3) and (R4) to obtain the uniqueness 
	 of geodesics (Propositions~\ref{prop:P-sequence=>arch} and \ref{prop:formula_for_modularsemilattice}).
	 \item[(P4)] Consider the general case of $x$ and $y$.
	 We choose a suitable element $a$ 
	 for which $a \vee x$ and $a \vee y$ are in the case of (P2) for ${\cal F}(a)$.
	 According to (P1), we obtain the unique geodesic $Q$ between 
	 $a \wedge x$ and $a \wedge y$ in $K({\cal I}(a))$.
	 According to (P3), we obtain the unique geodesic $R$ between 
	 $a \vee x$ and $a \vee y$ in $K({\cal F}(a))$.
	 Finally, with the help of Lemma~\ref{lem:decomposition}~(3) and Lemma~\ref{lem:product_median},
	 we combine $Q,R$ to a geodesic $P$ between $x,y$ and show its uniqueness. 
\end{itemize}
 In Section~\ref{subsec:lemmas}, we prove several necessary lemmas.
 In Section~\ref{subsec:geodesics},
 we complete the proof of Theorem~\ref{thm:unique_geodesic} 
 according to (P3) and (P4).

 In the following, all sublattices and subsemilattices satisfy (RP) 
 (or can be chosen so to satisfy (RP)), and 
 the corresponding subcomplexes are subspaces of the original space.
 In all cases, the verification of (RP) is straightforward and hence omitted.
\subsection{Lemmas}\label{subsec:lemmas} 
Let ${\cal L}$ be a modular semilattice.
Let $a$ be an element of ${\cal L}$.
Let ${\cal J}(a)$ denote
the set of elements $q$ such that $a$ and $q$ have join.
Define map $\omega_a:{\cal L} \to {\cal J}(a)$ by 
\[
\omega_a(p) := \mbox{the maximum element in ${\cal I}(p)$ having join with  $a$}.
\]
Then $\omega_a(p)$ is well-defined. 
Indeed, if  $p',p'' \preceq p$ have the join with $a$, 
then, by (LFL), their join $a \vee p' \vee p''$ exists 
(since $p' \vee p'' \preceq p$ exists). 
We also observe that $\omega_a(p) = p$ if and only if $p$ has join with $a$, 
i.e., $\omega_a$ is a retraction to ${\cal J}(a)$.
\begin{Lem}\label{lem:J(a)}
	For an element $a \in {\cal L}$, we have the following:
	\begin{itemize}
		\item[{\rm (1)}] ${\cal J}(a)$ is a modular subsemilattice of ${\cal L}$.
		\item[{\rm (2)}] $\omega_a$ is a nonexpansive order-preserving retraction to ${\cal J}(a)$.
		\item[{\rm (3)}] Subspace $K({\cal J}(a))$ is strictly-convex.
	\end{itemize}
\end{Lem}
\begin{proof}
	(1). 
	It is obvious that $p,q \in {\cal J}(a)$ implies $p \wedge q \in {\cal J}(a)$.
	Thus ${\cal J}(a)$ is a subsemilattice of ${\cal L}$.
    Also ${\cal J}(a)$ is join-closed, i.e., $p,q \in {\cal J}(a)$ and $p \vee q \in {\cal L}$ 
    imply $p \vee q \in {\cal J}(a)$ (by (LFL)). 	
	Clearly $p \in {\cal J}(a)$ implies ${\cal I}(p) \subseteq {\cal J}(a)$.
   Hence every principal ideal of ${\cal J}(a)$ is a modular lattice.
	Suppose that $p \vee q$, $q \vee r$, and $r \vee p$ have join with $a$.
	By (LFL) (and the join-closedness), necessarily $p \vee q \vee r$ exists.
	Then $p \vee a$, $p \vee q$, and $r \vee p$ have pairwise joins.
	Hence $p \vee q \vee r \vee a$ exists, and $p \vee q \vee r \in {\cal J}(a)$.
	This means that ${\cal J}(a)$ is a modular semilattice.

	(2). %
	Take a covering pair $(p,q)$.
	Obviously $\omega_a (p) \preceq \omega_a (q)$. Hence $\omega_a$ is order-preserving.
	Here $\omega_a(q) \wedge p$ and $a$ have join (since $\omega_a(q) \vee a$ is a common upper bound).
	Thus $\omega_a(q) \wedge p \preceq \omega_a(p) \preceq \omega_a(q) = \omega_a(q) \wedge q$.
	Since $r[\omega_a(q) \wedge p, \omega_a(q) \wedge q]$ is 
	at most $1$ (by modularity), 
	we have $r[\omega_a(p),\omega_a(q)] \leq 1$, and the nonexpansiveness of $\omega_a$.
	
	(3). Let $u,v \in K({\cal J}(a))$.
	Take a polygonal path $P$ connecting $u$ and $v$ 
	such that $P$ meets $K({\cal L}) \setminus K({\cal J}(a))$.
	By (2) and Lemma~\ref{lem:strictly_decrease}, the extension 
	$\omega_a: K({\cal L}) \to K({\cal J}(a))$ 
	is a nonexpansive retraction, and hence $d(\omega_a(P)) \leq d(P)$ holds.
	We show the strict inequality.
	We can take two points $x,y$ in $P$ 
	such that the segment $[x,y]$ belongs to a common simplex, 
	$x \in K({\cal J}(a))$, and $y \not \in K({\cal J}(a))$.
	Suppose that $x = \sum_{i} \lambda_i p_i$ and $y = \sum_i \lambda'_i p_i$.
	For some index $k$, it necessarily hold that $p_j \vee a$ exists for $j \leq k$, 
	$p_{j} \vee a$ does not exist for $j > k$, $\sum_{j: j> k} \lambda_i = 0$, 
	and $\sum_{j: j> k} \lambda'_i \neq 0$.
	Also $p_k = \omega_a(p_k) = \omega_a(p_{k+1})$.
	By Lemma~\ref{lem:strictly_decrease}~(2), it holds 
	$d(\omega_a(x), \omega_a(y)) < d(x,y)$, and hence  $d(\omega_a(P)) < d(P)$.
	Thus $K({\cal L}(a))$ is strictly-convex.
\end{proof}

A classical theorem by Dedekind and Birkhoff is that 
for any two chains in a modular lattice
there is a distributive sublattice containing them; 
see \cite[Theorem 3.18]{Birkhoff} and \cite[Theorem 363]{Gratzer}.
The next lemma is viewed as a generalization of this result. 
\begin{Lem}\label{lem:apartment}
	Let $p,q \in {\cal L}$.
	For four maximal chains 
	$\varPi \subseteq {\cal I}(p)$, $\varSigma \subseteq {\cal I}(q)$, $\varPi' \subseteq [p \wedge q, p]$,  
	and $\varSigma' \subseteq [p \wedge q, q]$, 
	there are distributive sublattices ${\cal B}$ of ${\cal I}(p)$
	and ${\cal C}$ of ${\cal I}(q)$ satisfying the following properties:
	\begin{itemize}
		\item[{\rm (1)}] ${\cal B}$ contains $\varPi$, $\varPi'$, and $(p \wedge q) \wedge \varSigma$. 
		\item[{\rm (2)}] ${\cal C}$ contains $\varSigma$, $\varSigma'$, and $(p \wedge q) \wedge \varPi$.
		\item[{\rm (3)}] ${\cal B} \cap {\cal C} = {\cal B} \cap {\cal I}(p \wedge q) = {\cal C} \cap {\cal I}(p \wedge q)$.
	\end{itemize}
\end{Lem}

\begin{proof}
	Let $s = r[p \wedge q, p]$ and $t = r[p \wedge q,q]$.
	We use the induction on $s + t$.
	Suppose that $s = t = 0$.
	Then $p = q = p \wedge q$, and $\varPi' = \varSigma'= \{ p \wedge q \}$.
	This case reduces to the original Dedekind--Birkhoff theorem; 
	but the following argument is easily adapted to prove this case.
	
	Suppose that $s > 0$. Suppose 
	that $\varPi = (0 = p_0 \prec p_1 \prec \cdots \prec p_n = p)$ with $n = r(p)$ 
	and $\varPi' = (p \wedge q = u_0 \prec u_1 \prec \cdots \prec u_s = p)$. 
	Consider chain $\tilde \varPi' := (u_0 \prec u_1 \prec \cdots \prec u_{s-1})$, and 
	chain $\tilde \varPi$ consisting of $p'_i := u_{s-1} \wedge p_i$ for $i=0,1,2\ldots,n$.
	Apply the induction on $\tilde \varPi, \tilde \varPi', \varSigma, \varSigma'$.
	We obtain distributive sublattices $\tilde {\cal B}$ of ${\cal I}(u_{s-1})$ 
	and ${\cal C}$ of ${\cal I}(q)$ 
	such that $\tilde {\cal B}$ contains $\tilde \varPi,(p \wedge q) \wedge \varSigma, \tilde \varPi'$, 
	${\cal C}$ contains $\varSigma, \varSigma', (p \wedge q) \wedge \tilde \varPi = (p \wedge q) \wedge  \varPi$, and 
	$\tilde {\cal B} \cap {\cal C} = \tilde{B} \cap {\cal I}(p \wedge q) = {\cal C} \cap {\cal I}(p \wedge q)$.
	
	We extend $\tilde {\cal B}$ to ${\cal B}$ 
	so that ${\cal B}$ contains $u_s =p$ and $\varPi$.
	Choose the smallest index $j$ such that $p_j \not \preceq u_{s-1}$.
	Since $u_{s-1}$ is covered by $p_n = u_s$, by modularity, 
	$p'_{k}$ is covered by $p_{k}$ for $k \geq j$.
	In particular, $p_{j-1} = p'_{j}$, and $p_{k}$ is equal to $p'_{k} \vee p_{j}$.
	Define ${\cal B}$ by
	\begin{equation}\label{eqn:B_vee_p_j}
	{\cal B} := \tilde{\cal B} \cup (\tilde {\cal B} \vee p_j) = \tilde {\cal B} \cup \{ v \vee p_j  \mid v \in \tilde {\cal B}: p_{j-1} \preceq v \preceq u_{s-1} \},
	\end{equation}
	where the second equality follows from 
	$w \vee p_j = w \vee p_{j-1} \vee p_j$ and $w \vee p_{j-1} \in \tilde {\cal B}$ 
	with $p_{j-1} \preceq w \vee p_{j-1} \preceq u_{s-1}$.
	Now $\{ v \vee p_j  \mid v \in \tilde {\cal B}: p_{j-1} \preceq v \preceq u_{s-1} \}$ 
	is isomorphic 
	to the interval between $p_{j-1}$ and $u_{s-1}$ in $\tilde{\cal B}$.
	Then, for $w \in {\cal B} \setminus \tilde{\cal B}$,
	there is unique $w' \in \tilde{\cal B}$ such that $p_{j-1} \preceq w' \preceq u_{s-1}$
	and $w = w' \vee p_{j}$.
	Then it holds
	\begin{equation}\label{eqn:a_wedge_w}
	a \wedge w = a \wedge w' \quad ( a \in \tilde {\cal B}).
	\end{equation}
	Indeed, $w$ covers $w'$ and $a \vee w$ covers $a \vee w'$.
	This implies $r[a \wedge w, a] = r[w,a \vee w] = r[w', a \vee w'] = r[a \wedge w',a]$.
	Then, by $a \wedge w' \preceq a \wedge w$, the equality must hold.
	In particular ${\cal B} (\supseteq \tilde {\cal B})$ is a sublattice of ${\cal I}(p)$ 
	containing $u_s$ and $\varPi$.
	
    It is easy but bit tedious to verify that ${\cal B}$ satisfies 
    the distributive law.
    Take $a,b,c \in {\cal B}$. 
    Suppose, e.g., that $c \in \tilde {\cal B} \not \ni a,b$.
    Then $(a \wedge c) \vee (b \wedge c) = (a' \wedge c) \vee (b' \wedge c) 
    = (a' \vee b') \wedge c = (a \vee b) \wedge c$, 
    where we use (\ref{eqn:a_wedge_w}) for the first and last equalities and 
    the distributive law in ${\cal B}$ for the second.
    For the last equality, 
    we use the fact that map $u \mapsto u \vee p_j$ is an isomorphism 
    from $[p_{j-1},u_{s-1}]$ to $[p_{j},u_{s}]$ in ${\cal B}$, 
    i.e., $(a \vee b)' = a' \vee b'$.   
    Similarly,
    $(a \vee c) \wedge (b \vee c) = (p_j \vee a' \vee c) \wedge (p_j \vee b' \vee c) = 
    p_j \vee ((a' \vee c) \wedge (b' \vee c)) = p_j \vee (a' \wedge b') \vee c = 
    (a \wedge b) \vee c$. 
    Suppose, e.g., that $b, c \not \in \tilde {\cal B} \ni a$.
    Then $(a \wedge c) \vee (b \wedge c) = (a \wedge c') \vee (b' \wedge c') \vee p_j = (a \wedge b') \vee c' \vee p_j = (a \wedge b) \vee c$, and
     $(a \vee c) \wedge (b \vee c) = (a \vee c' \vee p_j) \wedge (b' \vee c' \vee p_j) = ((a \vee c') \wedge (b' \vee c')) \vee p_j = (a \wedge b') \vee c' \vee p_j = (a \wedge b) \vee c$, 
     where from $a \vee c', b' \vee c' \in [p_{j-1},u_{s-1}]$  
     we use the isomorphic property of $u \mapsto u \vee p_j$ in the second equality.  
    The verifications for other cases are similar (more easy).
    By construction, the property (3) obviously holds. 
\end{proof}
The essence of the proof of Theorem~\ref{thm:modularlattice} in \cite{CCHO14} is 
the following.
\begin{Lem}[{\cite[Lemma 7.13]{CCHO14}}]\label{lem:modularlattice}
	Let ${\cal M}$ be a modular lattice and let ${\cal D}$ a distributive sublattice of ${\cal M}$. 
	Suppose that the join-irreducible elements $b_1,b_2,\ldots,b_n$ of ${\cal D}$ are arranged so that
	\[
	0 \prec b_1 \prec b_1 \vee b_2 \prec b_1 \vee b_2 \vee b_3 \prec \cdots \prec b_1 \vee \cdots \vee b_n
	\]
	is a maximal chain in ${\cal M}$.
	Then the map
	\begin{equation}\label{eqn:map}
	p \mapsto \{ b_i \mid  i: p \wedge (b_1 \vee b_2 \vee \cdots \vee b_{i}) \succ 
	p \wedge (b_1 \vee b_2 \vee \cdots \vee b_{i-1}) \} 
	\end{equation}
	is an order-preserving nonexpansive map from ${\cal M}$ to $\overline{\cal D} = 2^{\{b_1,\ldots,b_n\}}$
	such that it is identity on ${\cal D}$. 
	Consequently, $K({\cal D})$ is a (strictly-)convex subspace of $K({\cal M})$.
\end{Lem}
In this lemma, ${\cal D}$ is viewed as a sublattice of ${\cal M}$ as well as of $\overline{\cal D}$.
Recall Example~\ref{ex:distributivelattice} (or Proposition~\ref{prop:median_semilattice}) that $K({\cal D})$ 
is realized as a convex polytope in $[0,1]^{\{b_1,\ldots,b_n\}}$.
Then the unique geodesic in $K({\cal M})$ 
connecting two points $x,y$ is obtained as follows:
\begin{itemize}
	\item Choose a distributive sublattice ${\cal D}$ of ${\cal M}$ 
	containing two chains $\supp x$ and $\supp y$.
	\item Realize $K({\cal D})$ 
	as a convex polytope in $[0,1]^{n}$ and represent $x,y$ in the b-coordinate.
	\item $t \mapsto (1-t)x + ty$ is a geodesic between $x$ and $y$.  
\end{itemize}

Next we introduce $I(p,q)$ as the metric interval of $p,q$.
The {\em covering graph} of ${\cal L}$ is 
the undirected graph on 
${\cal L}$ such that each pair $\{u,v\}$ has an edge  
if and only if $(u,v)$ or $(v,u)$ is a covering pair. 
Let $d_{\cal L}$ denote the shortest path 
metric of the covering graph of ${\cal L}$.
For $p,q \in {\cal L}$, define $I(p,q)$ by
\begin{equation*}
I(p,q) := \{ u \in {\cal L} \mid d_{\cal L}(p,q)  = d_{\cal L}(p,u) + d_{\cal L}(u,q) \}.
\end{equation*}
For two subsets ${\cal B}, {\cal C} \subseteq {\cal L}$, 
let ${\cal B} \vee {\cal C}$ denote the set of elements $u$ that 
is represented as $u= b \vee c$ for some $b \in {\cal B}$ and $c \in {\cal C}$. 
\begin{Lem}[{\cite[Lemma 2.15]{HH0ext}}]~\label{lem:I(p,q)}
	For $p,q \in {\cal L}$, we have the following.
	\begin{itemize}
		\item[{\rm (1)}] $d_{\cal L}(p,q) = r[p \wedge q, p] + r[p \wedge q, q]$.
		\item[{\rm (2)}] $
		I(p,q) = [p \wedge q,p] \vee  [p \wedge q,p]$.
		\item[{\rm (3)}] If $u = b \vee c$ for 
		$b \in [p \wedge q, p], c \in [p \wedge q, q]$,
		then $b = p \wedge u$ and $c = q \wedge u$.
		\item[{\rm (4)}] $I(p,q)$ is a  modular subsemilattice in ${\cal L}$.
		For $u,u' \in I(p,q)$, it holds $u \wedge u' =
		(u \wedge u' \wedge p) \vee (u \wedge u' \wedge q)$.
	\end{itemize}
\end{Lem}
As mentioned in (P2), 
every geodesic connecting $x \in K({\cal I}(p))$ and $y \in K({\cal I}(q))$ with $p \wedge q = 0$
belongs to $K(I(p,q))$. We prove this fact in Proposition~\ref{prop:P-sequence=>arch}.
In the case of median semilattice, 
one can directly derive this fact by showing that 
map $u \mapsto (p \wedge u) \vee (q \wedge u)$ 
is an order-preserving nonexpansive retraction 
(that actually corresponds to the projection in (R1)). 
However this map is not nonexpansive for general modular semilattices. 

We extend the concept of an arch in $I(p,q)$.
An {\em arch} between $p$ and $q$ 
is a sequence $(p = u_0,u_1,u_2,\ldots,u_m = q)$ in $I(p,q)$ such that
\begin{equation}\label{eqn:arch_I(p,q)}
u_i \wedge p \succ u_j \wedge p, \quad u_i \wedge q \prec u_j \wedge q \quad (0 \leq  i < j \leq m).
\end{equation}
Observe that this actually generalizes 
the definition (\ref{eqn:arch}) of an arch in 
Section~\ref{subsec:cubical}.
\begin{Lem}\label{lem:K(A)}
	Let ${\cal A} = (p = u_0,u_1,u_2,\ldots,u_m = q)$ be an arch between $p$ and $q$.
	\begin{itemize}
		\item[{\rm (1)}] $\tilde {\cal I}({\cal A}):= \bigcup_{i=0}^m {\cal I}(u_i)$ 
		is a modular subsemilattice, 
		and $K(\tilde {\cal I}({\cal A}))$ is a CAT(0) subspace of $K({\cal L})$.
		\item[{\rm (2)}] The same holds for ${\cal I}({\cal A}) := \bigcup_{i=0}^m {\cal I}(u_i) \cap I(p,q)$.
		\item[{\rm (3)}] If $p \wedge q = 0$, then there is an order-preserving nonexpansive retraction
		from ${\tilde {\cal I}}({\cal A})$ to ${\cal I}({\cal A})$.
	\end{itemize}	
\end{Lem}
\begin{proof}
	(1). 
	Let ${\cal I}^k := \bigcup_{i=0}^k {\cal I}(u_i)$ for $k=0,1,2,\ldots,m$.
	We show that
	\[
	{\cal I}(u_{k+1}) \cap {\cal I}^k = {\cal I}(u_{k+1} \wedge u_k).
	\]
	Observe that $(\supseteq)$ is obvious. We show the converse.
	Take $v \in {\cal I}(u_{k+1}) \cap {\cal I}^k$.
	Then $v$ belongs to ${\cal I}(u_{k+1}) \cap {\cal I}(u_j) = {\cal I}(u_{k+1} \wedge u_j)$ for some $j \leq k$.
	By Lemma~\ref{lem:I(p,q)}~(4) and (\ref{eqn:arch_I(p,q)}), 
	we have $u_{k+1} \wedge u_j = (u_{k+1} \wedge u_j \wedge p) \vee (u_{k+1} \wedge u_j \wedge q) = 
	(u_{k+1} \wedge p) \vee (u_{j} \wedge q) \preceq (u_{k+1} \wedge p) \vee (u_{k} \wedge q) = u_{k+1} \wedge u_k$.
	Hence $v$ belongs to ${\cal I}(u_{k+1} \wedge u_k)$.
	
	Suppose by induction that ${\cal I}^k$ is a modular semilattice, 
	and $K({\cal I}^k)$ is CAT(0); the base case $k=0$ follows from Theorem~\ref{thm:modularlattice} and the fact that ${\cal I}(p)$ is a modular lattice. 
	Then ${\cal I}^{k+1}$ is a {\em gated amalgam} of modular semilattices ${\cal I}^k$ and ${\cal I}(u_{k+1})$ along {\em gated sub(semi)lattice} 
	${\cal I}(u_{k+1} \wedge u_k)$ 
	in the sense of \cite[Section 7.1]{CCHO14}.
	Then ${\cal I}^{k+1}$ is also a modular semilattice, 
	and $K({\cal I}^{k+1})$ is CAT(0) by \cite[Proposition~7.5]{CCHO14}.

	(2) follows from (1) by taking ${\cal L}$ as $I(p,q)$.
	
	(3). Take an arbitrary $v \in \tilde{\cal I}({\cal A})$.
	Let $\underline v$ be the maximum element in  ${\cal I}({\cal A})$ less than or equal to $v$, which is well-defined and given by
	\[
	\underline{v} = (v \wedge p) \vee (v \wedge q).
	\]
	Indeed, if $v \succeq w \in I(p,q)$, then 
	$w = (w \wedge p) \vee (w \wedge q) \preceq (v \wedge p) \vee (v \wedge q) = \underline{v}$.
	Let $\overline v$ be the minimum element in ${\cal I}({\cal A})$ greater than or equal to $v$; it is also well-defined since
	$I(p,q)$ is a subsemilattice.
	Then $\underline v \preceq v \preceq \overline{v}$, and 
	$v \in {\cal I}({\cal A})$ if and only if $\underline v = v = \overline{v}$.
	 Define $\varphi(v)$ by 
	 \begin{equation}
	 \varphi(v) := (\underline{v} \wedge p) \vee (\overline{v} \wedge q) \in I(p,q).
	 \end{equation}
	 Then $\overline{v} \preceq u_i$ implies $\varphi(v) \preceq u_i$.
	 This gives rise to a map $\varphi: \tilde{\cal I}({\cal A}) \to {\cal I}({\cal A})$.
	 
	We show that $\varphi$ is an order-preserving nonexpansive retraction.
	\begin{figure} 
		\begin{center} 
			\includegraphics[scale=0.7]{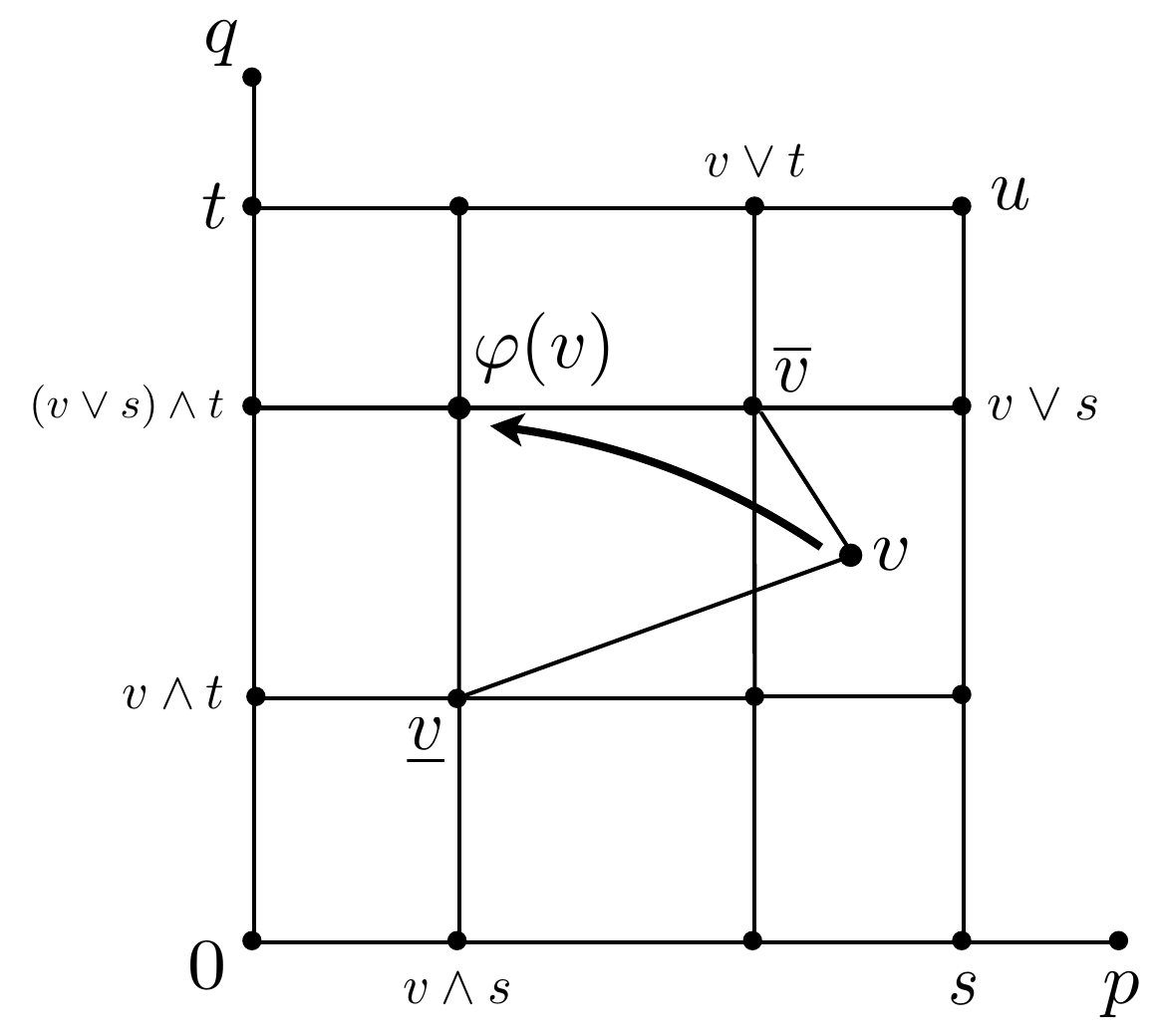}
			\caption{Construction of $\varphi$}  
					\label{fig:varphi}         
		\end{center}
	\end{figure}
    Figure~\ref{fig:varphi} 
    may be helpful to understand $\varphi$ and the following argument.  
	Here $\varphi$ is actually a retraction 
	by $\varphi(v) = (v \wedge p) \wedge (v \wedge q) = v$ if $\underline{v} = v = \overline{v}$.
	Take $u \in {\cal I}({\cal A})$ with  $v \preceq u$.
	Suppose that $u = s \vee t$ for
	$s  :=  u \wedge p$ and $t := u \wedge q$.
	Then we have 
	\begin{equation}\label{eqn:bar_v}
	\underline{v} =  (v  \wedge s) \vee (v \wedge t), \quad
	\overline{v}  = (v \vee s) \wedge (v \vee t).
	\end{equation}
	Indeed, if $w \preceq v$ and $w \in I(p,q)$, 
	then $w = (w \wedge p) \vee (w \wedge q) = (w \wedge s) \vee (w \wedge t) \preceq (v \wedge s) \vee (v \wedge t)$;
	notice $w \preceq u$ implies $w \wedge p \preceq u \wedge p = s$ and $w \wedge p = w \wedge p \wedge s = w \wedge s$. This gives the first equality.
	Dually, suppose that $v \preceq w$ and $w \in I(p,q)$.
	Then $v \preceq u \wedge w = (u \wedge w \wedge p) \vee (u \wedge w \wedge q) = (u \wedge w \wedge s) \vee (u \wedge w \wedge t) \in I(s,t)$. 
	Thus $(v \vee s) \wedge (v \vee t) \preceq 
	((u \wedge w) \vee s) \wedge ((u \wedge w) \vee t) 
	= ( s \vee (u \wedge w \wedge t)) \wedge  ((u \wedge w \wedge s) \vee t)
	= (u \wedge w \wedge s) \vee (u \wedge w \wedge t)
	= u \wedge w \preceq w$, where we use the calculation rule (Lemma~\ref{lem:I(p,q)}~(4)) in $I(s,t)$, 
	such as $(u \wedge w) \vee s = ((u \wedge w \wedge s) \vee (u \wedge w \wedge t)) \vee s = s \vee (u \wedge w \wedge t)$.  
	
	In particular, it holds
	\begin{equation}
	\underline{v} \wedge s = v \wedge s, \quad \underline{v} \wedge t = v \wedge t, 
	\quad \overline{v} \vee s = v \vee s, \quad \overline{v} \vee t = v \vee t.
	\end{equation}
	By (\ref{eqn:bar_v}), we have 
	\begin{equation}
	\overline{v} \wedge t = (v \vee s) \wedge (v \wedge t) \wedge t = (v \vee s) \wedge t.
	\end{equation}

	Take $v' \in \tilde {\cal I}({\cal A})$ that covers $v$.
	Here we can assume that $v,v' \preceq u$ (by retake $u$ if necessary).
	We show that $\varphi(v')$ covers $\varphi(v)$.
	
	Case 1: $v' \preceq \overline{v}$.
	Then $v' \vee t = v \vee t$ (by $v' \preceq \overline{v} \preceq \overline{v} \vee t = v \vee t$), and 
	$v' \vee s = v \vee s$.
	By the same argument for the proof of Lemma~\ref{lem:decomposition} (1), 
	this implies that $v' \wedge s$ covers $v \wedge s$
	and $v' \wedge t$ covers $v \wedge t$.
	Therefore $\varphi(v') = ((v' \vee s) \wedge t) \vee (v' \wedge s) = 
	((v \vee s) \wedge t) \vee (v' \wedge s)$, which covers $((v \vee s) \wedge t) \vee (v \wedge s) = \varphi(v)$.
	
	Case 2: $v' \not \preceq \overline{v}$.
	Then $v' \not \preceq \overline{v} \vee s = v \vee s$ and $v' \not \preceq \overline{v} \vee t = v \vee t$.
	This implies that $v' \vee s$ covers $v \vee s$ and $v' \vee t$ covers $v \vee t$, 
	and $v' \wedge s = v \wedge s$ and $v' \wedge t = v \wedge t$.
	Also $\overline{v'} \wedge t = (v' \vee s) \wedge t$ covers $(v \vee s) \wedge t$.
	Therefore $\varphi(v') = (v \wedge s) \vee ((v' \vee s) \wedge t)$ covers 
	$(v \wedge s) \vee ((v \vee s) \wedge t) = \varphi(v)$.	
\end{proof}

We call $K({\cal I}({\cal A}))$ the {\em path space} relative to an arch ${\cal A}$.
A geodesic in some path space is called a {\em path-space geodesic}.
The next lemma is used to reduce path-space geodesics in $K({\cal I}({\cal A}))$ to 
those in a median orthoscheme complex.
Recall Section~\ref{subsec:median} 
that $\overline{\cal D}$ is the Boolean extension of a median semilattice ${\cal D}$.
\begin{Lem}\label{lem:frame}
	Let ${\cal A} = (p = u_0,u_1,\ldots,u_m = q)$ be an arch, 
	and let $\varPi$ and $\varSigma$ be maximal chains in ${\cal I}(p)$ 
	and ${\cal I}(q)$, respectively. 
	There are distributive sublattices 
	${\cal B}$ of ${\cal I}(p)$ and  
	${\cal C}$ of ${\cal I}(q)$ satisfying the following properties:
	\begin{itemize}
		\item[{\rm (1)}]   
		${\cal B}$ contains $\varPi$ and $(p \wedge q) \wedge \varSigma$. 
		\item[{\rm (2)}] ${\cal C}$ contains $\varSigma$ and $(p \wedge q) \wedge \varPi$.
		\item[{\rm (3)}] ${\cal B} \vee {\cal C}$ contains ${\cal A}$. 
		\item[{\rm (4)}] 
	    ${\cal B} \vee {\cal C}$ is a median subsemilattice represented by a semi-bipartite PIP.
		\item[{\rm (5)}] 
		If $p \wedge q = 0$, then there is a nonexpansive order-preserving map from 
		${\cal I}({\cal A})$ to $\overline{({\cal B} \vee {\cal C})}$ 
		such that it is identity on $({\cal B} \vee {\cal C})({\cal A})  := {\cal I}({\cal A}) \cap ({\cal B} \vee {\cal C})$. 
	\end{itemize}
\end{Lem}
\begin{proof}
	Consider chains $\varPi' := 
	(p \wedge q = u_m \wedge p \prec u_{m-1} \wedge p \prec \cdots \prec u_0 \wedge p = p)$
	and $\varSigma' := 
	(p \wedge q = u_0 \wedge q \prec u_{1} \wedge q \prec \cdots \prec u_m \wedge q = q)$.
	By Lemma~\ref{lem:apartment}, 
	we can take distributive sublattices ${\cal B}$ of ${\cal I}(p)$ 
	and ${\cal C}$ of ${\cal I}(q)$ 
	such that ${\cal B}$ contains $\varPi$, $\varPi'$, $(p \wedge q) \wedge \varSigma$, and
	${\cal C}$ contains $\varSigma$, $\varSigma'$, $(p \wedge q) \wedge \varPi$.
	Then ${\cal A} \subseteq \varPi' \vee \varSigma' \subseteq {\cal B} \vee {\cal C}$.

	A semi-bipartite PIP representing ${\cal B} \vee {\cal C}$ is constructed as follows.
	Let $B$ and $C$ be the sets of 
	join-irreducible elements of distributive lattices ${\cal B}$ and ${\cal C}$, respectively.
	By property (3) in Lemma~\ref{lem:apartment},
	$B \cap C$ is the set of join-irreducible elements in ${\cal B}　\cap {\cal C} = {\cal I}(p \wedge q)$.
	Define PIP $G_{\preceq}$ on vertex set $B \cup C$, 
	where the partial order is the restriction of ${\cal L}$ 
	and an edge is given to each unbounded pair.  
	Then it is easy to verify that ${\cal B} \vee {\cal C}$ is isomorphic to ${\cal S}(G_{\preceq})$:
	For $p \in {\cal B} \vee {\cal C}$, consider the set of join-irreducible elements $a \preceq p$, 
	which is a stable ideal in $G_{\preceq}$. Conversely, for a stable ideal, 
	consider the join of all elements in the stable ideal, 
	which exists by (LFL) and belongs to ${\cal B} \vee {\cal C}$.
	Observe that $G_{\preceq}$ is a semi-bipartite PIP with tri-partition 
	$B',C',Z$, where $B' = B \setminus Z$ and $C' = C \setminus Z$ 
	for some $Z \supseteq B\cap C$.

	Next we define a nonexpansive order-preserving map 
	$\phi: {\cal I}({\cal A}) \to \overline{({\cal B} \vee {\cal C})}$ to show~(5).
	Suppose that $B = \{b_1,b_2,\ldots,b_k\}$ and $C = \{c_1,c_2,\ldots,c_l\}$, where $B \cap C = \emptyset$ 
	by $p \wedge q = 0$.
	Consider a maximal chain of ${\cal I}(p)$ containing $u_i \wedge p$ $(i=0,1,\ldots,m)$.
	We can assume that this chain is given by $b_1 \vee \cdots \vee b_i$ $(i=0,1,\ldots,k)$.
	Similarly we can assume that $c_1 \vee \cdots \vee c_i$ $(i=0,1,\ldots,l)$
	is a maximal chain containing $u_i \wedge q$ $(i=0,1,\ldots,m)$.
    Define map $\phi_B:{\cal I}(p) \to \overline{\cal B} = 2^{B}$ 
    by (\ref{eqn:map}) for the chain $b_1 \vee \cdots \vee b_i$. 
    By Lemma~\ref{lem:modularlattice}, 
    $\phi_B$ is a nonexpansive order-preserving map fixing ${\cal B}$.
    Define $\phi_C:{\cal I}(q) \to \overline{\cal C}$ for the chain $c_1 \vee \cdots \vee c_i$ similarly.
    Now define map $\phi:{\cal I}({\cal A}) \to \overline{({\cal B} \vee {\cal C})}$ by
	\[
	\phi(u) := \phi_{B}(s) \cup \phi_{C}(t) \quad (u = s \vee t \in {\cal I}({\cal A}):s \in {\cal I}(p), t \in {\cal I}(q)).
	\]
	Notice from Lemma~\ref{lem:I(p,q)}~(2) and (3) 
	that the expression $u= s \vee t$ is possible and unique.
	We need to verify that $\phi_{B}(s) \cup \phi_{C}(t)$ belongs to 
	$\overline{({\cal B} \vee {\cal C})}$.
	Indeed, from $s \vee t \preceq u_i$ for some $i$, it holds 
	that $s \preceq u_i \wedge p$ and $t \preceq u_i \wedge q$.
	Since the above chain contains $u_i \wedge p$, we have 
	$\phi_B(s) \subseteq \phi_B(u_i \wedge p)$. 
	Similarly we have
	$\phi_C(t) \subseteq \phi_C(u_i \wedge q)$.
	This means that $\phi_B(s) \cup \phi_C(t)$ 
	consists of join-irreducible elements in ${\cal I}(u_i)$,   
	and belongs to $\overline{({\cal B} \vee {\cal C})}$ (since there is no edge among them).
	By construction, $\phi$ fixes $({\cal B} \vee {\cal C})({\cal A})$.
	The nonexpansiveness of $\phi$ 
	follows from that of $\phi_B$ and $\phi_C$ (Lemma~\ref{lem:modularlattice}).
\end{proof}
We call the above ${\cal B} \vee {\cal C}$ a {\em distributive frame}.

\subsection{Geodesics}\label{subsec:geodesics}
The goal of this subsection is 
to prove the unique geodesic property (Theorem~\ref{thm:unique_geodesic}).
For $x \in K({\cal L})$, denote by $\tau(x)$ the maximum element in 
$\supp x$, i.e., if $x = \sum_{i=0}^k \lambda_i p_i$ with $\lambda_k \neq 0$ 
then $\tau (x) = p_k$.
Equivalently, $\tau(x)$ is the minimum element $u$ 
such that $K({\cal I}(u))$ contains $x$.

We start with a general property of geodesics in $K({\cal L})$.
\begin{Prop}\label{prop:P-sequence}
	Let $P$ be a geodesic in $K({\cal L})$.
    Then there are $0= t_0 < t_1 < \cdots < t_{m+1} = 1$ $(m \geq 0)$ 
    and $u_0,u_1,\ldots,u_m \in {\cal L}$ satisfying the following: 
	\begin{itemize}
		\item[{\rm (1)}] $\tau(P(t)) = u_i$ for $t_{i} < t < t_{i+1}$ and $0 \leq i \leq m$.
		\item[{\rm (2)}] $\tau(P(0)) \preceq u_0$, $\tau(P(1)) \preceq u_m$, 
		and $\tau(P(t_i)) \preceq u_{i-1} \wedge u_{i}$ for $1 \leq i \leq m$.
		\item[{\rm (3)}] The join of $u_{i-1}$ and $u_{i}$ does not exist for $1 \leq i \leq m$. 
	\end{itemize}
\end{Prop}

\begin{proof}
	Let $t \in (0,1)$. Choose a sufficiently small $\epsilon > 0$. 
	Then $P(t - \epsilon)$ and $P(t)$ belong a common simplex of chain $C$ and 
	$P(t)$ and $P(t + \epsilon)$ belong to a common simplex of chain $C'$.
	Let $x := P(t)$, $x^{-} := P(t - \epsilon)$, and $x^{+} := P(t + \epsilon)$.
	Let $u := \tau (x)$, $u^{-} := \tau (x^{-})$, 
	and $u^{+} := \tau (x^{+})$.
	Necessarily $u^-$ and $u$ are comparable, and $u^+$ and $u$ are comparable.
	Now $\epsilon > 0$ is small.
	By continuity of $P$, both $u^+ \prec u$ and $u^- \prec u$ are impossible.  
	(By this argument, 
	we also see that $\tau (P(0)) \preceq \tau (P(\epsilon))$ and $\tau (P(1)) \preceq \tau (P(1-\epsilon))$.) 

	We next show that $u^- = u \prec u^+$ cannot occur.
	Suppose to the contrary that $u^- = u \prec u^+$ holds.
	Then the chains $C$ and $C'$ belong to a modular lattice ${\cal I}(u^+)$. 
	According to the Dedekind--Birkhoff theorem (or Lemma~\ref{lem:apartment}), 
	there is a distributive sublattice ${\cal B}$ of ${\cal I}(u^+)$ 
	such that $K({\cal B})$ contains $C,C'$.
	Suppose that $x^-,x, x^+$ are represented in the b-coordinate as
	\begin{equation}\label{eqn:x^-}
	x^- = \sum_{i} \mu_i^- b_i,\  x = \sum_{i} \mu_i b_i, \ x^+ = \sum_{i} \mu_i^+ b_i,
	\end{equation}
	where $b_i$ are join-irreducible elements of ${\cal B}$.
	By $u^- = u \prec u^+$, 
	at least one of $\mu_i$ is zero, and all of $\mu_i^+$ are nonzero.
	However $x$ is the midpoint of $x^-$ and $x^+$ in convex polytope $K({\cal B}) \subseteq [0,1]^n$.
	Then $\mu_i = (\mu_i^- + \mu_i^+)/2$ must hold. This is a contradiction.
	Similarly, $u^- \prec u = u^+$ is also impossible.
	Thus $u^- = u = u^+$ or $u^- \succ u \prec u^+$ holds.

	Suppose that $u^- \succ u \prec u^+$ holds.
	We show that the join of $u^-$ and $u^+$
	does not exist.
	Suppose not.
	Then the chains $C$ and $C'$ belong to a modular lattice 
	${\cal I}(u^- \vee u^+)$.
	Choose a distributive sublattice ${\cal B}$ of ${\cal I}(u^+ \vee u^-)$ 
	such that $K({\cal  B})$ contains $C,C'$.
	Represent $x^-,x, x^+$ in the b-coordinate as (\ref{eqn:x^-}).
	By $u^- \succ u \prec u^+$,  
	it must hold $\{i \mid \mu^-_i > 0\} \supset \{i \mid \mu_i > 0\} \subset \{i \mid \mu_i^+ > 0\}$.
	However this is a contradiction since $\mu_i = (\mu_i^- + \mu_i^+)/2$ must hold.

	Since any geodesic meets a finite number of simplices, 
	we conclude the existence of $u_i$ and $t_i$ with (1-3).	
\end{proof}
The sequence $(u_0,u_1,\ldots,u_m)$ determined by a geodesic $P$ is called the {\em $P$-sequence}.

We say that $x,y \in K({\cal L})$ are {\em orthogonal}
if for $p = \tau(x)$ and $q= \tau(y)$ 
it hold $\omega_q(p)= \omega_p(q) = p \wedge q = 0$.
From now, 
let us fix an orthogonal pair $x,y \in K({\cal L})$.
Suppose that $p = \tau (x)$ and $q = \tau (y)$.
We study geodesics connecting $x,y$.
The following shows that any geodesic connecting $x,y$ 
must be a path-space geodesic, which establishes~(R3).
\begin{Prop}\label{prop:P-sequence=>arch}
	\begin{itemize}
		\item[{\rm (1)}] For a geodesic $P$ 
		connecting $x$ and $y$, 
		the $P$-sequence is an arch~${\cal A}$ for $I(p,q)$; in particular, $P$ belongs to $K(\tilde {\cal I}({\cal A}))$ and $K({\cal I}({\cal A}))$.
		\item[{\rm (2)}] For two geodesics $P,P'$ connecting $x$ and $y$, 
		if the $P$-sequence and the $P'$-sequence are equal, 
		then $P$ and $P'$ are equal.
	\end{itemize}
\end{Prop}
\begin{proof}
	(1). Let $(u_0,u_1,\ldots,u_m)$ be the $P$-sequence.
	Define $p_i$ $(i=0,1,\ldots,m)$ by $p_0 := p \wedge u_0$ and
		\begin{equation*}
		p_{i} := p_{i-1} \wedge u_i \wedge p = p_{i-1} \wedge u_i \quad (i=1,2,\ldots,m). 
		\end{equation*}
		Then $p = p_0 \succeq p_1 \succeq \cdots \succeq p_m = 0$, 
		where $p = p_0$ follows from Proposition~\ref{prop:P-sequence}~(2) and
		$p_m = 0$ follows from the orthogonality and $q \preceq u_m$ (Proposition~\ref{prop:P-sequence}~(2)). 
	Similarly, define $q_i$ $(i=0,1,\ldots,m)$ by $q_m := q \wedge u_m$ and
	\begin{equation*}
	q_{i} := q_{i+1} \wedge u_i \wedge q = q_{i+1} \wedge u_i \quad (i=1,2,\ldots,m).  
	\end{equation*}
	Then $q = q_m \succeq q_{m-1} \succeq \cdots \succeq q_0 = 0$.
	Observe that $p_i$ and $q_i$ have upper bound $u_i$ and have join.
	For $i=0,1,2,\ldots,m$, 
	let $\bar u_i := p_i \vee q_i$, which belongs to $I(p,q)$ (Lemma~\ref{lem:I(p,q)}~(2)).
	See Figure~\ref{fig:P-sequence} 
	for construction of $\bar u_0, \bar u_1,\ldots, \bar u_m$.
	\begin{figure} 
		\begin{center} 
			\includegraphics[scale=0.7]{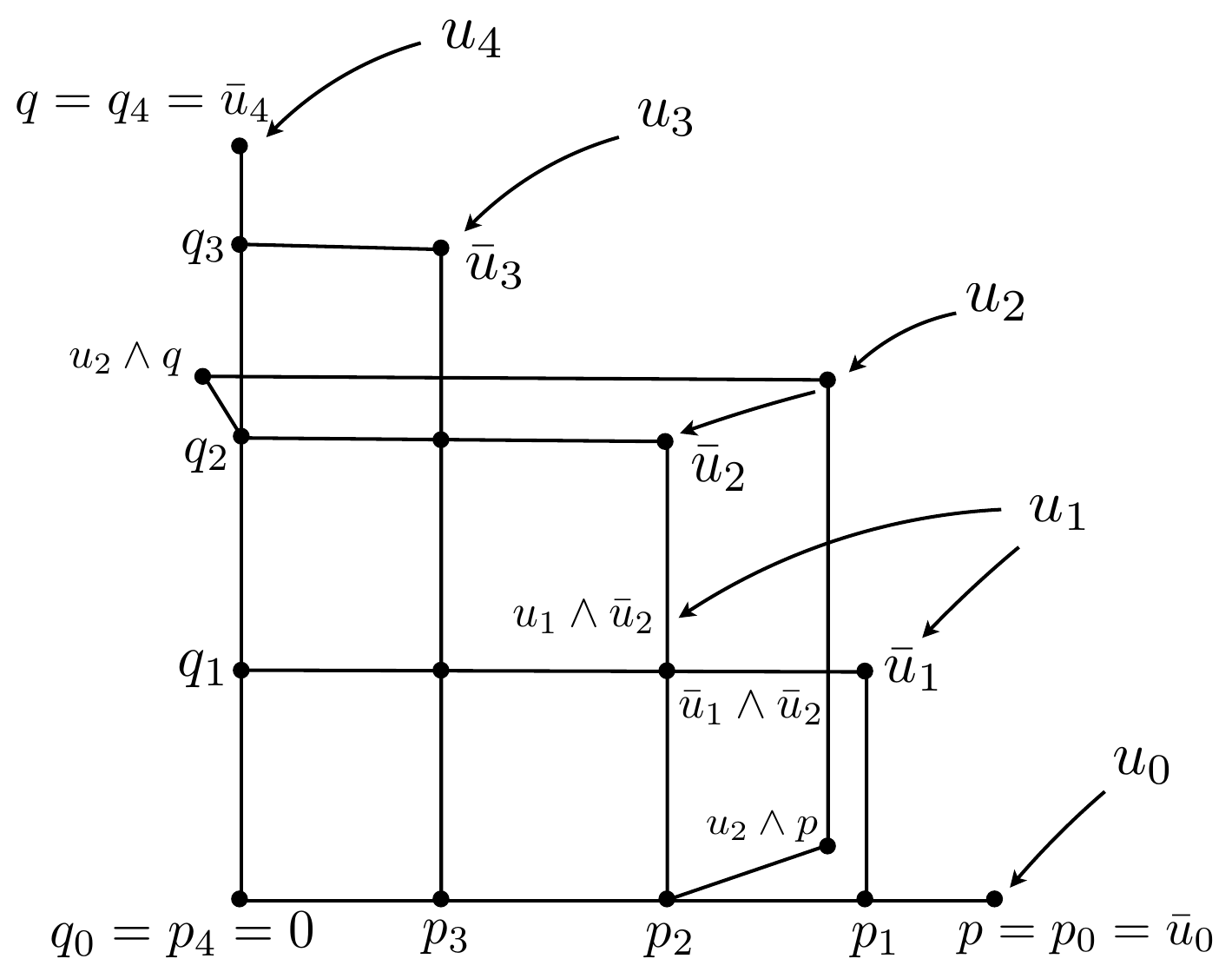}
			\caption{Construction of an arch}  
			\label{fig:P-sequence}         
		\end{center}
	\end{figure}
	
	Let $0 = t_0 < t_1 < \cdots < t_{m+1} =1$ be 
	the moments for which $\tau(P(t_i)) \preceq u_{i-1} \wedge u_i$ for $i=1,2,\ldots, m+1$.
	We are going to show that $P':[0,1] \to K(\tilde{\cal I}({\cal A}))$ defined by
	\begin{equation*}
	P'(t) := 
	\bar u_i \wedge P(t) \quad (t \in [t_i,t_{i+1}], i=0,1,\ldots,m)
	\end{equation*}
	is a well-defined $(x,y)$-path in $K(\tilde{\cal I}({\cal A}))$.
	Notice from $\bar u_0 = p$ and $\bar u_m = q$ 
	that $P'(0) = x$ and $P'(1) = y$. 	
	We have to show that
	\begin{equation}\label{eqn:1}
	\bar u_{i-1} \wedge P(t_i) = \bar u_{i} \wedge P(t_{i}) \quad (i=1,2,\ldots,m).
	\end{equation} 
	We show 
	\begin{equation}\label{eqn:2}
	\bar u_{i-1} \wedge u_i = u_{i-1} \wedge \bar u_i \quad (i=1,2,\ldots,m).
	\end{equation}
	Since $\tau(P(t_i)) \preceq u_{i-1} \wedge u_i$,
	if (\ref{eqn:2}) is true, 
	then $\bar u_{i-1} \wedge P(t_i) = \bar u_{i-1} \wedge u_i \wedge P(t_i) 
	= u_{i-1} \wedge \bar u_i \wedge P(t_i) =  \bar u_i \wedge P(t_i)$, which implies (\ref{eqn:1}).
	
	By $\bar u_{i-1} \preceq u_{i-1}$, 
	it obviously holds $\bar u_{i-1} \wedge \bar u_{i} \preceq u_{i-1} \wedge \bar u_i \preceq \bar u_{i}$.
	Observe from $p_{i} \preceq p_{i-1}$, $q_{i-1} \preceq q_i$ and Lemma~\ref{lem:I(p,q)} that
	$d_{\cal L}(p,q) = d_{\cal L}(p, \bar u_{i-1} \wedge \bar u_i) 
	+ d_{\cal L}(\bar u_{i-1} \wedge \bar u_i, \bar u_{i}) + d_{\cal L}(\bar u_{i}, q)$.
	Consequently, it holds
	$\bar u_{i} \wedge u_{i-1} \in [\bar u_{i-1} \wedge \bar u_{i}, \bar u_{i}] \subseteq I(p,q)$.
	Thus we have
	\begin{eqnarray*}
	\bar u_{i} \wedge u_{i-1} & = & (\bar u_{i} \wedge u_{i-1} \wedge p) \vee (\bar u_{i} \wedge u_{i-1} \wedge q) \\
	&= & ( (p_i \vee q_i) \wedge (u_{i-1} \wedge p)) \vee ( (p_i \vee q_i) \wedge (u_{i-1} \wedge q)) \\
	& = & p_{i} \vee q_{i-1} = \bar u_{i-1} \wedge \bar u_{i}, 
	\end{eqnarray*}
	where we use the calculation rule in Lemma~\ref{lem:I(p,q)}~(4) with $p_i \preceq p_{i-1} \preceq u_{i-1} \wedge p$ and $q_{i-1} = q_{i} \wedge u_{i-1}$.
	Similarly, from $u_{i} \wedge \bar u_{i-1} \in [\bar u_i \wedge \bar u_{i-1}, \bar u_{i-1}] \subseteq I(p,q)$, 
	we obtain $u_{i} \wedge \bar u_{i-1} = \bar u_{i-1} \wedge \bar u_{i}$, 
	and hence (\ref{eqn:2}).

	For each $i$
	the map $\bar u_i \wedge$ is nonexpansive (Lemma~\ref{lem:decomposition2}).
	Then it must hold $u_i = \bar u_i$ for $i=0,1,2,\ldots,m$.
	Indeed, if $P$ leaves $K({\cal I}(\bar u_i))$ at $t_i$, 
	then one can deduce, by precisely the same argument in the proof of Lemma~\ref{lem:decomposition}~(2), 
	contradiction $d(P') < d(P)$.
	Moreover, 
	by Proposition~\ref{prop:P-sequence}~(3), it must hold 
	$u_i \wedge p = p_i \succ p_{i+1} = u_{i+1} \wedge p$ and
	$u_i \wedge q = q_i \prec q_{i+1} = u_{i+1} \wedge q$ for $i=0,1,2,\ldots,m-1$.
	Thus $(p= u_0,u_1,\ldots,u_m = q)$ is an arch.
	
	By Lemma~\ref{lem:K(A)}~(1-3),
	$K({\cal I}({\cal A}))$ is a strictly-convex subspace of $K(\tilde{\cal I}({\cal A}))$, and hence
	$P$ must belongs to $K({\cal I}({\cal A}))$.
	
	(2). The path space is CAT(0) (Lemma~\ref{lem:K(A)}) and is uniquely geodesic. 
	Since $P,P'$ belong to the same path space, 
	it must hold $P= P'$. 
\end{proof}

By (2), our problem reduces to find an arch 
for which the corresponding path-space geodesic is shortest. 
We give an explicit formula 
of path-space geodesics,
which is naturally obtained via a distributive frame 
introduced in Lemma~\ref{lem:frame}.
Let ${\cal A} = (p = u_0,u_1,\ldots,u_m = q)$ be an arch. 
Take a distributive frame ${\cal B} \vee {\cal C}$ for $I(p,q)$ 
containing $\supp x$, $\supp y$, and ${\cal A}$.
By $p \wedge q = 0$, ${\cal B} \vee {\cal C}$ is represented 
by a bipartite PIP $G =(V,E,\preceq)$ with color classes $B$ and $C$, 
where $B = \supp x$ and $C = \supp y$ in the b-coordinate of $x,y$.
Now elements $u$ in ${\cal B} \vee {\cal C}$ corresponds to stable ideal 
$U = \{ v \in B \cup C \mid v \preceq u \}$.
In this correspondence, arch ${\cal A} = (p = u_0,u_1,\ldots,u_m = q)$
is the arch $(B = U_0, U_1, \ldots, U_m = C)$ in 
the sense of Sections~\ref{subsec:cubical} and \ref{subsec:median}. 
Via the b-coordinates of $x,y$, 
define $X_i$, $Y_i$, $\|X_i \|$, $\|Y_i\|$ as in Section~\ref{subsec:cubical}.
By Lemma~\ref{lem:b-coordinate}, 
the quantities $\|X_i\|$ and $\|Y_i\|$ are written as
\begin{eqnarray*}
\|X_i\|^2 & = & d(x \wedge u_{i-1})^2 - d(x \wedge u_{i})^2,  \\
\|Y_i\|^2 & = & d(y \wedge u_i)^2 - d(y \wedge u_{i-1})^2, 
\end{eqnarray*}
where $d(z)$ is defined by
\begin{equation*}
d(z) := d(0,z) = \sqrt{\sum_{v \in V} z_v^2 } \quad (\mbox{if }z =_{\rm b} \sum_{v \in V} z_v v).
\end{equation*}
Accordingly, $v({\cal A};x,y)$ is defined by
\begin{eqnarray*}
&& v({\cal A};x,y)^2 := \sum_{i=1}^m (\|X_i\| + \|Y_i\|)^2 \\  
&& =  \sum_{i=1}^m \left(\sqrt{d(x \wedge u_{i-1})^2 - d(x \wedge u_{i})^2} + \sqrt{d(y \wedge u_i)^2 - d(y \wedge u_{i-1})^2} \right)^2. 
\end{eqnarray*}
Also the $(x,y)$-concavity of an arch is defined by (\ref{eqn:concave}). 
These notions are independent of the choice of a distributive frame.
Thus we have:
\begin{Prop}\label{prop:formula_for_modularsemilattice}
	For an $(x,y)$-concave arch ${\cal A} = (p =u_0,u_1,\ldots,u_m =q)$ 
	and any distributive frame ${\cal B} \vee {\cal C}$ containing ${\cal A}$, $\supp x$, and $\supp y$,
	the unique geodesic connecting $x,y$ in $K({\cal I}({\cal A}))$ belongs to $K({\cal B}\vee {\cal C})$, 
	and is given by (\ref{eqn:formula}). 
	Its length is equal to $v({\cal A};x,y)$.
	Moreover, any path-space geodesic belongs 
	to the path-space for some $(x,y)$-concave arch.  	
\end{Prop}
\begin{proof}
The statement follows from:
\begin{itemize}
	\item The path space $K({\cal I}({\cal A}))$ 
	is uniquely geodesic (Lemma~\ref{lem:K(A)}~(2)).
	\item There is a nonexpansive map 
	from $K({\cal I}({\cal A}))$ to $K(\overline{{\cal B} \vee {\cal C}})$
	fixing $K(({\cal B} \vee {\cal C})({\cal A}))$ (Lemma~\ref{lem:frame}~(5)).
	\item Any path-space geodesic connecting $x,y$ in $K(\overline{{\cal B} \vee {\cal C}})$
	belongs to $K({\cal B} \vee {\cal C})$, and is the path-space geodesic 
	in the path space for some $(x,y)$-concave arch (Proposition~\ref{prop:formula_for_mediansemilattice}).
\end{itemize}
\end{proof}
Finally we prove (R4) that the minimum of $v({\cal A};x,y)$ 
over all arches ${\cal A}$ is uniquely attained 
by some arch. 
Define $\xi = \xi_{x,y}:I(p,q) \to \RR^2$ by
\begin{equation*}
\xi(u) = (d(x \wedge u)^2, d(y \wedge u)^2) \quad (u \in I(p,q)).
\end{equation*}
Recall Figure~\ref{fig:arch} and the argument before Remark~\ref{rem:MSSP}.
Consider the convex hull of $\xi(u) \in \RR^2$ for all $u \in I(p,q)$, which is denoted by $\Conv I(p,q)$.
Observe that $\Conv I(p,q)$ belongs to the square 
$[0, d(x)^2] \times [0, d(y)^2]$ and
has $(0,0)$, $(d(x)^2,0)$, $(0, d(y)^2)$
as extreme points.
Consider elements $u \in I(p,q)$ mapped to extreme points of $\Conv I(p,q)$ other than zero $(0,0)$. 
This method was introduced by \cite{HH0ext}.
\begin{Prop}\label{prop:canonical_arch}
	\begin{itemize}
	\item[{\rm (1)}] The set of elements in $I(p,q)$ mapped to nonzero extreme points by $\xi$
	is arranged to be an $(x,y)$-concave arch ${\cal A}^*$.
	\item[{\rm (2)}] The arch ${\cal A}^*$ uniquely attains $\min_{\cal A} v({\cal A}; x,y)$.
	In particular, a geodesic in $K(I(p,q))$ connecting $x,y$ is 
	unique and is the path-space geodesic in $K({\cal A}^*)$.
	\end{itemize}  
\end{Prop}

In the proof of (1), the following property has a key role.
\begin{Lem}\label{lem:supermodular}
	Let ${\cal M}$ be a modular lattice.
	For $x \in K({\cal M})$, 
	the function $a \mapsto d(x \wedge a)^2$ is supermodular, i.e.,
	\begin{equation}\label{eqn:supermodular}
	d(x \wedge a)^2 + d(x \wedge b)^2 \leq d(x \wedge (a \wedge b))^2 + d(x \wedge (a \vee b))^2
	\quad (a,b \in {\cal M}).
	\end{equation}
	In addition, if $\tau(x)$ is the maximum element in ${\cal M}$, 
	then $a \mapsto d(x \wedge a)^2$ is monotone increasing, i.e.,
	\begin{equation*}
	d(x \wedge a)^2 < d(x \wedge b)^2 \quad (a,b \in {\cal M}: a \prec b)
	\end{equation*}
\end{Lem}
\begin{proof}
	Suppose that $x = \sum_{i=0}^n \lambda_i p_i$ 
	for a maximal chain $p_i$ $(i=0,1,2,\ldots,n)$.
	Let $J(a)$ denote the set of indices $i (>0)$ with $p_i \wedge a \succ p_{i-1} \wedge a$.
	Then, as in the proof of Lemma~\ref{lem:strictly_decrease},
	it holds
	\begin{equation}\label{eqn:d(xva)^2}
	d(x \wedge a)^2 = \sum_{i \in J(a)} (\lambda_i + \cdots + \lambda_n)^2. 
	\end{equation}
	Here it holds
	\[
	J(a) \subseteq J(b) \quad (a,b: a \preceq b).
	\]
	Indeed, suppose that $i \in J(a)$ and $b$ covers $a$. Recall the argument of the proof of Lemma~\ref{lem:decomposition}~(1).
	Since $p_i \wedge a$ covers $p_{i-1} \wedge a$, 
	it holds $p_{i-1} \vee a = p_{i} \vee a$.
	If $b \preceq p_{i-1} \vee a = p_{i} \vee a$, 
	then $b \wedge p_{i-1}$ covers $a \wedge p_{i-1}$ and  
	 $b \wedge p_{i}$ covers $a \wedge p_{i}$, 
	 which implies that $b \wedge p_{i}$ covers  $b \wedge p_{i-1}$, i.e., $i \in J(b)$.
	 If $b \not \preceq p_{i-1} \vee a = p_{i} \vee a$, 
	 then $b \wedge p_i = a \wedge p_{i}$ and $b \wedge p_{i-1} = a \wedge p_{i-1}$, which also implies $i \in J(b)$.
	  
	Also $|J(a)|$ is equal to the rank of $a$; indeed,
	consider the chain consisting of $a \wedge p_i$s, which is a maximal chain from $0$ to $a$.
	This means that $J(a) \subset J(b)$ if $a \prec b$.
	Thus if $\tau(x)$ is the maximum element in ${\cal M}$, 
	then $\lambda_n > 0$, and, by (\ref{eqn:d(xva)^2}), $a \mapsto d(x \wedge a)^2$ is monotone increasing.
	
	We next show the supermodularity.
	By the standard argument, 
	it suffices to show the supermodular modularity inequality (\ref{eqn:supermodular})
	for those pairs $(a,b)$ for which $a$ and $b$ cover $a \wedge b$ and are 
	covered by $a \vee b$; see, e.g., the proof of \cite[Proposition 3.8]{HH0ext}.
	Let $J := J(a \wedge b)$.
	For some indices $i,j \not \in J$ with $i < j$, 
	it holds $J(a \vee b) = J + i + j$.
	Then $(J(a), J(b))$ is equal to $(J + i, J + i)$, 
	$(J + i, J + j)$, $(J + j, J + i)$, or $(J + j, J + j)$.
	We show that the first case $(J + i, J + i)$ cannot occur.
	Suppose that $J(a) = J+i$.
	Then $p_i \wedge a$ covers $p_{i-1} \wedge a$.
	Here ($*$) $p_{i-1} \wedge a = p_{i-1} \wedge a \wedge b$ holds.
	This follows from the fact that  
	$p_i \wedge a$ covers or equals 
	$p_i \wedge a \wedge b$ and covers $p_{i-1} \wedge a \succeq  p_{i-1} \wedge a \wedge b = p_i \wedge a \wedge b$ (by $i \not \in J)$.
	Thus $a = (p_i \wedge a) \vee (a \wedge b)$.
	Next, let $w := p_{j-1} \wedge a \in [p_i \wedge a, a]$.
	Then $a = w \vee (a \wedge b)$ holds (by $w \in [p_i \wedge a, a]$ and $a = (p_i \wedge a) \vee (a \wedge b)$).
	By the same argument for ($*$), it holds $w = p_{j-1} \wedge (a \vee b)$. 
	Namely $a$ is determined by $p_{j-1}$, $a \wedge b$, and $a \vee b$.
	If $J(b) = J+ i$, then by the same argument we have $b = w \vee (a \wedge b) = a$; 
	this is a contradiction to $a \neq b$.
	
	Thus $(J(a), J(b))$ is equal to 
	$(J + i, J + j)$, $(J + j, J + i)$, or $(J + j, J + j)$.
	For the first and second cases, by (\ref{eqn:d(xva)^2}) we have
	\begin{eqnarray*}
	d(x \wedge a)^2 + d(x \wedge b)^2 
	& =& 2 d(x \wedge (a \wedge b))^2 + (\lambda_i + \cdots + \lambda_n)^2 + (\lambda_j + \cdots + \lambda_n)^2 \\ 
	&= & d(x \wedge (a \wedge b))^2 + d(x \wedge (a \vee b))^2.
	\end{eqnarray*}
	For the last case, we have
	\begin{eqnarray*}
		d(x \wedge a)^2 + d(x \wedge b)^2 & = & 2 d(x \wedge (a \wedge b))^2 + 
		(\lambda_j + \cdots + \lambda_n)^2 + (\lambda_j + \cdots + \lambda_n)^2 \\
		& \leq & 2 d(x \wedge (a \wedge b)) + (\lambda_i + \cdots + \lambda_n)^2 + (\lambda_j + \cdots + \lambda_n)^2 \\
		& \leq & d(x \wedge (a \wedge b))^2 + d(x \wedge (a \vee b))^2.
	\end{eqnarray*}
\end{proof}

\begin{proof}[Proof of Proposition~\ref{prop:canonical_arch}~(1)]
	Choose $u,u' \in I(p,q)$ such that $\xi(u)$ 
	and $\xi(u')$ are nonzero extreme points in $\Conv I(p,q)$.
	Suppose that $\xi(u) = \xi(u')$ or 
	$\xi(u)$ and $\xi(u')$ 
	are adjacent extreme points with 
	$d(x \wedge u) \leq d(x \wedge u')$ and $d(y \wedge u) \geq d(y \wedge u')$ 
	(with at least one of the inequality being strict).
	We show that $u=u'$ for the former case and  
	$u \wedge p \prec u' \wedge p$ and $u \wedge q \succ u' \wedge q$ for the latter case,
	which implies the statement.
	
	Let $(a,b) := (u \wedge p, u \wedge q)$ and $(a',b') := (u' \wedge p, u' \wedge q)$.
	All pairs $(a,  a')$, $(a,  b \wedge b')$,  
	and $(a', b \wedge b')$ among the
	triple $(a, a', b \wedge b')$
	are bounded.
	By (LFL), 
	their join
	$v := (a \vee a') \vee (b \wedge b')$ 
	exists and belongs to $I(p,q)$ (Lemma~\ref{lem:I(p,q)}). 
	Similarly the join $v' := (a \wedge a' ) \vee (b \vee b')$ 
	of triple $(a \wedge a', b, b')$
	exists and belongs to $I(p,q)$.
	Therefore $\xi(v) =
	( d(x \wedge (a \vee a'))^2,  d(y \wedge (b \wedge b'))^2)$ and 
	$\xi(v') = 
	( d(x \wedge (a \wedge a'))^2,  d(y \wedge (b \vee b'))^2$. 
	By supermodularity (Lemma~\ref{lem:supermodular}), we have
	\begin{eqnarray*}
		&& \xi(v) + \xi(v')  = 	( d(x \wedge (a \vee a'))^2 + d(x \wedge (a \wedge a'))^2,  
		d(y \wedge (b \wedge b'))^2 +  d(y \wedge (b \vee b'))^2) \\
		&& \geq ( d(x \wedge a)^2 + d(x \wedge a')^2,  
		d(y \wedge b)^2 +  d(y \wedge b')^2) \\
		&& = \xi(u) + \xi(u').
	\end{eqnarray*}
	Then both $\xi(v)$ and $\xi(v')$ must belong to 
	$[\xi(u), \xi(u')]$ since it is an edge or an extreme point of $\Conv I(p,q)$.
	By $d(x \wedge (a \vee a'))^2 \geq d(x \wedge a)^2$ and
	$d(y \wedge (b \vee b'))^2 \geq d(y \wedge b)^2$, 
	it must hold
	$\xi(v)  = \xi(u)$
	and $\xi(v') = \xi(u')$.
	Thus $(d(x \wedge (a \vee a')), d(y \wedge (b \wedge b'))) = (d(x \wedge a), d(y \wedge b))$ and 
	$(d(x \wedge (a \wedge a')), d(y \wedge (b \vee b'))) = (d(x \wedge a'), d(y \wedge b'))$. 
	By  $p = \tau(x)$, $q = \tau (y)$, and  Lemma~\ref{lem:supermodular}, 
	the functions $d(x \wedge \cdot)$ and $d(y \wedge \cdot)$ 
	are monotone increasing.
	Then we have $(a \wedge a', a \vee a') = (a',a)$ and $(b \wedge b', b \vee b') = (b,b')$. 
    Therefore $a' \preceq a$ and $b \preceq b'$.
    If $\xi(u) = \xi(u')$, then $a' \succeq a$ and $b \succeq b'$ 
    also hold, and we have $u = a \vee b = a' \vee b' = u'$.
    Suppose that $\xi(u) \neq \xi(u')$. Then at least one of $a' \prec a$ and $b \prec b'$ holds.
    If $a' = a$ and $b \prec b'$ (say), 
    then it necessarily holds $a = a' = p \wedge u_0 = p$ and $u' = u_1 = p \vee b_1 (\succ 0)$; 
    however this contradicts the orthogonality of $p,q$. Therefore $a' \prec a$ and $b \prec b'$, as required.
\end{proof}

For a convex polygon $Q \subseteq [0, \kappa] \times [0, \lambda]$
containing $(0,0)$, $(\kappa,0)$, $(0, \lambda)$ (as extreme points),
define $v(Q) (\geq 0)$ by
\[
v(Q)^2 := \sum_{i=1}^{m} (\sqrt{\alpha_{i-1} - \alpha_{i}} + \sqrt{\beta_{i} - \beta_{i-1}})^2,
\]
where $(\kappa,0) = (\alpha_0,\beta_0), (\alpha_1,\beta_1),\ldots (\alpha_m, \beta_m) = (0, \lambda)$
are nonzero extreme points of $Q$ 
such that $(\alpha_i,\beta_i)$ and $(\alpha_{i+1},\beta_{i+1})$ are adjacent by an edge.
Now $v(\Conv I(p,q)) = v({\cal A}^*;x,y)$ with $\kappa = d(p)^2$ and $\lambda = d(q)^2$.
Then Proposition~\ref{prop:canonical_arch}~(2) follows from:
\begin{Lem}\label{lem:v(K)>v(K')}
	For two polygons $Q,Q' \subseteq [0, \kappa] \times [0, \lambda]$
	containing $(0,0)$, $(\kappa,0)$, $(0, \lambda)$,
	if $Q \subset Q'$ (proper inclusion), then $v(Q) > v(Q')$.
\end{Lem}
\begin{proof}
	Choose an edge of $Q$ joining nonzero extreme points $(\alpha,\beta)$ and $(\alpha',\beta')$, 
	and choose a point $(\alpha^*, \beta^*)$ in the interior the edge.
	Suppose that $\alpha > \alpha^* >\alpha'$ and $\beta' < \beta^* < \beta$. 
	Perturb $(\alpha^*, \beta^*)$ into outside of $Q$ 
	so that $Q \subset \tilde Q := {\rm Conv} (Q \cup \{ (\alpha^*, \beta^*)\})$. 
	The perturbation is sufficiently small. 
	The set of extreme points of $\tilde Q$ is obtained by adding $(\alpha^*, \beta^*)$ 
	to the set of extreme points of $Q$, where $(\alpha^*, \beta^*)$
	is adjacent to $(\alpha,\beta)$ and $(\alpha', \beta')$. 
	Then $v(\tilde Q)^2 - v(Q)^2$ is equal to 
	\begin{equation}\label{eqn:triangle}
	(\sqrt{\alpha - \alpha^*} + \sqrt{\beta^* - \beta})^2 + 
	(\sqrt{\alpha^* - \alpha'} + \sqrt{\beta' - \beta^*})^2 - (\sqrt{\alpha - \alpha'} + \sqrt{\beta' - \beta})^2.
	\end{equation}
	Now 
	the points $(\alpha, \beta)$, $(\alpha^*, \beta^*)$, and $(\alpha', \beta')$ in $\RR^2$ are not collinear. The non-collinearity 
	is equivalent to $(\alpha - \alpha^*)(\beta' - \beta^*) \neq (\alpha^* - \alpha')(\beta^* - \beta)$.
	This in turn implies that the points 
	$(\sqrt{\alpha - \alpha^*} + \sqrt{\beta^* - \beta},0)$, $(\sqrt{\beta^* - \beta}, \sqrt{\alpha^* - \alpha'})$, $(0, \sqrt{\alpha^* - \alpha'} + \sqrt{\beta' - \beta^*})$ are not collinear.
	By the triangle inequality for these three points, 
    (\ref{eqn:triangle}) is negative, and thus $v(\tilde Q) < v(Q)$.
	In this way, we can expand $Q$ until $Q = Q'$.
	Then $v$ is strictly decreasing.
	In the expansion, 
	we can remove extreme points with keeping $v(Q)$ when they become non-extreme.
\end{proof}

Now we are ready to prove the unique geodesic property of $K({\cal L})$ (Theorem~\ref{thm:unique_geodesic}).
\begin{proof}[Proof of Theorem~\ref{thm:unique_geodesic}]
	Let $x,y$ be arbitrary points in $K({\cal L})$.
	Let $p :=\tau (x)$ and $q := \tau (y)$.	
	Consider
	$a :=  (\omega_q(p) \vee q) \wedge (\omega_p(q) \vee p) = \omega_q(p) \vee \omega_p(q)$ (by Lemma~\ref{lem:I(p,q)}). 
	We can assume that $a \succ 0$, i.e., $x,y$ are not orthogonal.
	Consider $x \vee a$ and $y \vee a$, which belong to $K({\cal F}(a))$.
	Now $x \vee a$ and $y \vee a$ are orthogonal in $K({\cal F}(a))$, since 
	the minimum element of ${\cal F}(a)$ is $a$, 
	$\tau (x \vee a) = \omega_p(q) \vee p$, and $\tau (y \vee a) = \omega_q(p) \vee q$.
	By Proposition~\ref{prop:canonical_arch},
    a geodesic $Q$ 
    between $x \vee a$ and $y \vee a$ in $K({\cal F}(a))$ uniquely exists 
    and also belongs to the path space $K({\cal I}({\cal A}^*))$ for the arch ${\cal A}^*$ for 
    $I(\tau (x \vee a), \tau(y \vee a))$ (with minimum $a$).

    Consider a distributive frame 
    ${\cal B} \vee {\cal C}$ containing ${\cal A}^*$, $\supp x$, and 
    $\supp y$.
    The above $Q$ is a geodesic between $x \vee a$ and $y \vee a$ in 
    $K({\cal B} \vee {\cal C})$.
    Also $x \wedge a$ and $y \wedge a$ 
    belong to $K({\cal D})$ for the distributive sublattice ${\cal D} := ({\cal B} \vee {\cal C}) \cap {\cal I}(a)$ 
    of modular lattice ${\cal I}(a)$.
    Therefore there is a unique geodesic $R$ in 
    between $x,y$ in $K({\cal D})$, which must be a geodesic in $K({\cal I}(a))$ (Lemma~\ref{lem:modularlattice}).
    Now we have two geodesics $Q,R$ in $K({\cal B} \vee {\cal C})$, 
    where $Q$ connects $x \vee a$ and $y \vee a$ and 
    $R$ connects $x \wedge a$ and $y \wedge a$.
    Represent ${\cal B} \vee {\cal C}$ by a semi-bipartite PIP $G_{\preceq}$ 
    with tripartition $B',C',Z$.
    Consider paths $Q,R$ and points $x,y, x \vee a,y \vee a, x \wedge a, y \wedge a$ 
    in the b-coordinate. Notice that $a$ corresponds to $Z$ 
    and the coefficient of $z \in Z$ in $Q$ is always $1$.
    There is no pair $u \in Z, v \in B' \cup C'$ with $u \succ v$.
    This means that every path in $K(G'_{\preceq})$ can be lifted to 
    $K(G_{\preceq})$ by defining the coefficient of $z \in Z$ as $1$.
    Therefore the projection $Q' = Q|_{B' \cup C'}$ is 
    a unique geodesic in $K(G'_{\preceq})$ connecting $x|_{B' \cup C'}$ and $y|_{B' \cup C'}$.
    Also $R$ is a unique geodesic in $K(G^0_{\preceq})$ connecting $x|_{Z} = x \wedge a$ and $y|_{Z} = y \wedge a$.
    By Lemma~\ref{lem:product_median}, 
    we obtain a unique geodesic $P = (Q',R)$ in 
    $K(G_{\preceq}) = K({\cal D})$ connecting $x$ and $y$.
    By considering the original coordinates in $K({\cal L})$   
    we have $a \vee P(t) = Q(t)$ and $a \wedge P(t) = R(t)$.
    By Lemma~\ref{lem:decomposition} (3), 
    $P$ is a geodesic in $K({\cal L})$ connecting $x$ and $y$ and satisfies
    \begin{equation}\label{eqn:tight}
    d(P)^2 = d(a \vee P)^2 + d(a \wedge P)^2 
    = d(a \vee x, a \vee y)^2 + d(a \wedge x, a \wedge y)^2.
    \end{equation}
    
    Finally we show that the constructed geodesic $P$ is actually a unique geodesic.   
    Consider another geodesic $P'$ in $K({\cal L})$ connecting $x$ and $y$.
    By Lemma~\ref{lem:decomposition}~(3), the path $P'$ must satisfy the above equality (\ref{eqn:tight}).
    Namely $d(a \vee P') = d(a \vee x, a \vee y)$ must hold.
    By the uniqueness of a geodesic connecting 
    an orthogonal pair, 
    the images of $a \vee P'$ and $a \vee P$ are the same.
    Now $a \vee P'$ belongs to the path space $K({\cal I}({\cal A}^*))$.
    This implies that $P'$ also belongs to $K(\tilde{\cal I}({\cal A}^*))$.
  By the unique geodesic property of   $K(\tilde{\cal I}({\cal A}^*))$ (Lemma~\ref{lem:K(A)}~(1)), 
  it must hold $P = P'$, as required.  
\end{proof}

\begin{Rem}\label{rem:MVSP}
	As in the case of CAT(0) cubical complex or median orthoscheme complex
	(Remarks~\ref{rem:MSSP} and \ref{rem:MSIP}),
	a geodesic in $K({\cal L})$ can be obtained via 
	the arch ${\cal A}^*$.
	Again ${\cal A}^*$ is obtained by the following (parametric) optimization problem:
	\begin{eqnarray*}
	{\rm Max.} && (1 -\lambda) d(u \wedge x)^2 + \lambda d(u \wedge y)^2 \\
	{\rm s.t.} && u \in I(p,q) = {\cal I}(p) \vee {\cal I}(q).
	\end{eqnarray*}
	This problem is a far-reaching generalization 
	of the maximum weight stable set problem in a bipartite graph, 
	and includes {\em weighted maximum vanishing subspace problem (WMVSP)}~\cite{HamadaHirai} 
	as a special case where ${\cal L}$ is a modular semilattice of 
	vector subspaces on which each of given bilinear forms vanishes.
	WMVSP is viewed as a {\em submodular optimization on modular lattice}, 
	which is one of the current issues in combinatorial optimization.
	A polynomial time algorithm for WMVSP is not known in general, 
	and deserves a challenging open problem. 	
	See \cite[Section 5]{HiraiNakashima} for a polynomial solvable special case.
\end{Rem}

\section*{Acknowledgments}
The author thanks Koyo Hayashi 
for meticulously reading and numerous helpful comments.
This work was partially supported by JSPS KAKENHI Grant Numbers 
JP26280004, JP17K00029.

\appendix
\section{Appendix: Proof of Theorem~\ref{thm:owen}~(2)}
Let $K = K_m$ be the subspace of $\RR^m$ consisting of points $x = (x_1,x_2,\ldots,x_m)$ satisfying
\[
x_i < 0 \Rightarrow x_j \leq 0 \quad (1 \leq  j \leq i \leq m).
\]
Equivalently $K = \bigcup_{i=0}^m (\RR_{-}^{i} \times \RR_{+}^{m-i})$, 
where $\RR_-$ denotes the set of nonpositive reals. 
For $a,c \in K$ with $c \leq a$
let $K^{c,a}$ denote the intersection of $K$ and 
the box $[c_1,a_1] \times [c_2,a_2] \times \cdots \times [c_m,a_m]$.
\begin{Lem}\label{lem:K^ca}
	If $a > 0$ then
	$K^{c,a}$ is a strictly-convex subspace of $K$.
\end{Lem}
\begin{proof}
	Consider map $\varphi: \RR^m \to [c_1,a_1] \times [c_2,a_2] \times \cdots \times [c_m,a_m]$ defined by
	$\varphi(x)_i := x_i$ if $x_i \in [c_i,a_i]$, $c_i$ if $x_i \leq c_i$, and $a_i$ if $x_i \geq a_i$.
	It is easy to see that $\varphi$ is a strictly-nonexpansive retraction.
	Thus it suffices to show that $x \in K$ implies $\varphi(x) \in K$.
	Suppose $\varphi(x)_i < 0$. By $a_i > 0$, it holds $x_i < 0$ and $c_i < 0$.
	Therefore, by $x,c \in K$, 
	it holds $x_j \leq 0$ and $c_j \leq 0$ for $j \leq i$. 
	Thus $\varphi(x)_j \leq 0$ for $j \leq i$.
	This concludes $\varphi(x) \in K$.
\end{proof}

Consider the path space $K({\cal A})$ for an arch ${\cal A} = (B = U_0,U_1,\ldots,U_m = C)$, 
points $x,y \in K({\cal A})$ with $\supp x = B$ and $\supp y = C$, and notations in Section~\ref{subsec:cubical}.
Here $K({\cal A})$ is considered as an orthant space 
(by replacing $[0,1]$ with $\RR_{+}$).
Let $a \in \RR^m$ and $b  \in \RR^m$ be defined by
\begin{equation}
a_i := \|X_i\|, \quad  b_i := \|Y_i\| \quad (i \in [m]).
\end{equation}
Define map $\iota:K^{-b,a} \to K({\cal A})$ by
\begin{equation}
\iota (\lambda)_{X_i \cup Y_i}  
:= \left\{
\begin{array}{cc}
\lambda_i x_{X_i}/ a_i  & {\rm if}\ \lambda_i \geq 0,  \\
- \lambda_i y_{Y_i}/ b_i & {\rm if}\ \lambda_i \leq 0,
\end{array}\right. \quad (\lambda \in K^{-b,a}).
\end{equation}
Here we abbreviate $x|_{X_i}$ as $x_{X_i}$; this abbreviation is used in sequel.
\begin{Lem}[{\cite[Theorem 4.4]{Owen11}}]\label{lem:iota}
	The map $\iota$ embeds $K^{-b,a}$ into $K({\cal A})$ as a subspace, and
	every geodesic $P$ connecting $x,y$ in $K({\cal A})$
	belongs to $K^{-b,a} \subseteq K({\cal A})$.
\end{Lem}
\begin{proof}
	We observe the former statement by 
	$\|\iota(\lambda)_{X_i \cup Y_i} - \iota(\lambda')_{X_i \cup Y_i}\| 
	= \|\lambda_i x_{X_i}/ a_i - \lambda_i' x_{X_i}/ a_i\| 
	= |\lambda_i - \lambda'_i|$ (if $\lambda_i,\lambda'_i\geq 0$). 
	
	For each $i \in [m]$, consider the orthogonal projections $\varphi_i$ from $\RR^{X_i}$ 
	to $\RR x_{X_i}$ and $\varphi'_i$ from $\RR^{Y_i}$ to $\RR y_{Y_i}$.
	This gives rise to a retraction $\varphi :K({\cal A}) \to K({\cal A})$	
	defined by $\varphi(z)_{X_i} := \varphi_i(z_{X_i})$ and $\varphi(z)_{Y_i} := \varphi'_i(z_{Y_i})$
	for $i\in [m]$. Then $\varphi$ is strictly-nonexpansive, 
	and its image is viewed as $K_m$.
	Hence $P$ belongs to $K^{-b,a}$ (by Lemma~\ref{lem:K^ca}).
\end{proof}

Hence the geodesic problem on $K({\cal A})$ 
reduces to that on $K^{-b,a}$ for positive vectors $a,b \in \RR^n$.
Let $I = (I_1,I_2,\ldots,I_k)$ be an ordered partition of $[m]$, i.e.,
for some $0 = i_0 < i_1 < i_2 < \cdots < i_k = m$, 
it holds $I_\ell =[i_\ell] \setminus [i_{\ell-1}]$ for $\ell \in [k]$.
Let $K^{-b,a}_I$ be the subspace of $K^{-b,a}$
consisting of points $x$ satisfying
\begin{itemize}
	\item for each $j$, there is $\lambda \in [0,1]$ such that 
	$x_{I_j} = \lambda a_{I_j}$ or $x_{I_j} = - \lambda b_{I_j}$.
\end{itemize}
\begin{Lem}\label{lem:K^-ba_I}
	For an ordered partition $I=(I_1,I_2,\ldots,I_k)$ of $[m]$, the subspace 
	$K^{-b,a}_I$ is isometric to $K^{- b^I, a^I} \subseteq K_k$ for $a^I,b^I \in \RR^k$ defined by
	\[
	(a^I)_j :=\| a_{I_j} \|,\quad (b^I)_j :=\|b_{I_j} \| \quad (j \in [k]),
	\]
	where the isometry $\phi: K^{-b^I, a^I} \to K^{-b,a}_I$ is given by
	\[
	\phi(\lambda)_{I_j} := 
	\left\{ \begin{array}{cc}
	\lambda_j a_{I_j}/ \|a_{I_j}\|   & {\rm if}\ \lambda_j \geq 0, \\
	\lambda_j b_{I_j}/ \|b_{I_j}\|    & {\rm if}\ \lambda_j \leq 0,
	\end{array}\right.
	\quad (\lambda \in K^{-b^I, a^I}, j \in [k]).
	\]
\end{Lem}
\begin{proof}
	A straightforward verification similar to the proof of Lemma~\ref{lem:iota}.
\end{proof}

Next consider a geodesic $P:[0,1] \to K$ connecting $P(0) = a$ and $P(1) = -b$, 
which belongs to $K^{-b,a}$ (Lemma~\ref{lem:K^ca}).
For $i=1,2,\ldots,m$, 
let $t_i \in [0,1]$ denote the first time for which the $i$-th coordinate of $P$ is zero (nonpositive), and
let $p^i := P(t_i)$.
\begin{Lem}\label{lem:t1<t2<...<tm}
	\begin{itemize}
		\item[{\rm (1)}] $0< t_1 \leq t_2 \leq \cdots \leq t_m < 1$.
		\item[{\rm (2)}] $P(t)_i  > 0$ if $t < t_i$ and $P(t)_i < 0$ if $t > t_i$.
	\end{itemize}
\end{Lem}
\begin{proof}
	By Lemma~\ref{lem:K^ca}, 
	the subpath $P'$ from $t=0$ to $t_i$ 
	belongs to the strictly-convex subspace $K^{p^i,a}$ of $K^{-b,a}$.
	Since $p^i_j \geq 0$ must hold for $j \geq i$, 
	it holds $K^{p^i,a} = K^{p^i_{[i-1]},a_{[i-1]}} \times 
	[p^i_i,a_{i}] \times [p^i_{i+1},a_{i+1}] \times \cdots \times [p^i_m, a_m]$.
	Therefore (after re-parametrization) the path $P'$ can be written as
	the product of geodesics in $K^{p^i_{[i-1]},a_{[i-1]}}$
	and in the box $[p^i_i,a_{i}] \times \cdots \times [p^i_m, a_m]$ (Lemma~\ref{lem:product}).  
	The latter path is the segment between $a_{[m] \setminus[i-1]}$ and $p^{i}_{[m] \setminus[i-1]}$.
	This means that for $i' \geq i$ 
	the $i'$-th coordinate cannot become zero before the $i$-th coordinate becomes zero.
	Also the $i$-th coordinate is positive before $t_i$.
	
	By a similar way (or reversing time), 
	the $i$-th coordinate is negative after~$t_i$.
\end{proof}
Define the ordered partition $I = (I_1,I_2,\ldots, I_k)$ of $[m]$ 
such that $i$ and $i'$ belong to the same part if and only if $p^i= p^{i'}$.
\begin{Lem}[{\cite[Corollary 4.7]{Owen11}}]\label{lem:I1I2...Ik}
	The geodesic $P$ belongs to subspace $K^{-b,a}_I$, and hence is also a geodesic in $K^{-b,a}_I$. 
\end{Lem}
\begin{proof}
	For $j=1,2,\ldots,k$, consider the time $t := t_{i}$ for $i \in I_j$.
	As in the above proof, the subpath of $P$ from $0$ to $t$
	is the product of a path and the segment between $a_{J}$ and $p^i_{J}$  
	for $J := I_i \cup I_{i+1} \cup \cdots \cup I_{k}$.
	In particular, $P(t)_{I_i}$ is written as $\lambda(t) a_{I_i}$ for $t \leq t_i$.
	Consequently, $P$ belongs to $K^{-b,a}_I$.
\end{proof}

\begin{Lem}\label{lem:different}
	Suppose that $p^1,p^2,\ldots,p^m$ (or $t_1,t_2,\ldots,t_m$) are all different. 
	Then $P$ is the straight line between $a$ and $-b$, i.e.,
	\[
	P(t) = a - t (a + b) \quad (t \in [0,1]).
	\]
\end{Lem}
\begin{proof}
	By Lemma~\ref{lem:t1<t2<...<tm}, the image of $P$ is 
	$[p^0,p^1] \cup [p^1,p^2] \cup \cdots \cup [p^m,p^{m+1}]$ for $p^0 := a$ and $p^{m+1} :=-b$.
	For $i \in [m]$, choose 
	$q \in [p^{i-1},p^{i}] \setminus \{p^i\}$ and 
	$q' \in [p^{i},p^{i+1}] \setminus \{p^i\}$ that are sufficiently close to $p^i$.
	The sign patterns of $q$ and $q'$ 
	are $(- \cdots \stackrel{i-1}{-} \stackrel{i}{+} \stackrel{i+1}{+}  \cdots +)$
	and $(- \cdots \stackrel{i-1}{-} \stackrel{i}{-} \stackrel{i+1}{+}  \cdots +)$, respectively.
	Hence $[q,q']$ belongs to $K$.
	Necessarily $[q,q']$ is a part of $P$, and $P$ does not bend at $t_i$.
	Thus $P$ is a straight line.
\end{proof}
We are ready to prove Theorem~\ref{thm:owen}~(2).
Now suppose that $P$, $a$, and $b$ comes from $K({\cal A})$.
By Lemma~\ref{lem:I1I2...Ik}, 
$P$ is a geodesic in $K^{-b,a}_{I}$.
Via $K^{-b,a}_{I} \simeq K^{-b^{I},a^{I}}$ (Lemma~\ref{lem:K^-ba_I}), 
we can regard $P$ as a geodesic in $K^{-b^{I},a^{I}}$. 
Now points $p^1,p^2,\ldots,p^{k}$ (or $t_1,t_2,\ldots,t_{k}$) are different in $K^{-b^{I},a^{I}}$.  
Thus $P$ is a straight line in $K^{-b^{I},a^{I}}$.
In particular, $t_i = \|a_{I_i}\|/(\|a_{I_i}\| + \|b_{I_i}\|)$ for $i \in [k]$.
By $t_1 < t_2 < \cdots < t_k$, we obtain the concavity condition:
\[
\|a_{I_1}\|/\|b_{I_1} \| <  \|a_{I_2}\| / \|b_{I_2} \|  < \cdots < \|a_{I_k}\| / \|b_{I_k} \|.
\]
Returning the path space $K({\cal A})$, 
this means that the geodesic $P$ in $K({\cal A})$ 
is equal to the path-space geodesic (\ref{eqn:formula}) in $K({\cal A}')$ 
for an $(x,y)$-concave subarch ${\cal A'} = (B = U'_0,U'_1,\ldots,U'_k =C)$, 
where $U'_j = U_i$ for the last index $i$ in $I_{j-1}$. 
This proves Theorem~\ref{thm:owen}~(2). 
Moreover one can see from Lemma~\ref{lem:v(K)>v(K')} that 
this arch ${\cal A}'$ consists of members of ${\cal A}$ 
that corresponds to extreme points of the convex hull of $(0,0)$ and 
$\xi_k$ $(k=0,1,2,\ldots,m)$.

\end{document}